\documentclass[a4paper]{amsart}
\usepackage{amsmath,amsthm,amssymb,latexsym,epic,bbm,comment}
\usepackage{graphicx,enumerate,stmaryrd}
\usepackage[all,2cell]{xy}
\xyoption{2cell}

\usepackage[active]{srcltx}
\usepackage[parfill]{parskip}
\UseTwocells

\newtheorem{theorem}{Theorem}
\newtheorem{lemma}[theorem]{Lemma}
\newtheorem{remark}[theorem]{Remark}
\newtheorem{corollary}[theorem]{Corollary}
\newtheorem{proposition}[theorem]{Proposition}
\newtheorem{example}[theorem]{Example}

\newcommand{\tto}{\twoheadrightarrow}

\font\sc=rsfs10
\newcommand{\cC}{\sc\mbox{C}\hspace{1.0pt}}
\newcommand{\cG}{\sc\mbox{G}\hspace{1.0pt}}

\newcommand{\cI}{\sc\mbox{I}\hspace{1.0pt}}

\newcommand{\cS}{\sc\mbox{S}\hspace{1.0pt}}
\newcommand{\cT}{\sc\mbox{T}\hspace{1.0pt}}

\newcommand{\cU}{\sc\mbox{U}\hspace{1.0pt}}

\font\scc=rsfs7
\newcommand{\ccC}{\scc\mbox{C}\hspace{1.0pt}}

\begin{document}

\title[Transitive $2$-representations]{Transitive $2$-representations\\ of finitary $2$-categories}
\author{Volodymyr Mazorchuk and Vanessa Miemietz}

\begin{abstract}
In this article, we define and study the class of simple transitive $2$-representations of finitary $2$-categories. 
We prove a weak version of the classical Jordan-H{\"o}lder Theorem where the weak composition subquotients
are given by simple transitive $2$-representations. For a large class of finitary $2$-categories we prove that
simple transitive $2$-representations are exhausted by cell $2$-representations. Finally, we show that
this large class contains finitary quotients of $2$-Kac-Moody algebras.
\end{abstract}

\maketitle

\section{Introduction}\label{s0}

This article, for the first time, proves a general classification result for an axiomatically
defined class of $2$-representations of a large class of $2$-categories covering most examples
studied in the area of categorification.

More specifically, we study finitary $2$-categories over an algebraically closed field which 
include the $2$-category of Soergel bimodules associated to a finite Coxeter system (see \cite{BG,So,EW}),
an exhaustive family of quotients of $2$-Kac-Moody algebras (see \cite{BFK,KL,Ro,CL,We}),
quiver $2$-categories constructed in \cite{Xa} and the $2$-category of 
projective functors on the module category of a finite dimensional algebra (see \cite{MM1}).
We define a new class of $2$-representations for such $2$-categories which we call 
{\em simple transitive $2$-representations} and which we believe serves as the correct $2$-analogue
for the class of irreducible representations of an algebra. Our definition of 
simple transitive $2$-representations comes in two layers, the first being a discrete transitive
action of the multisemigroup of $1$-morphisms (this alone is called {\em transitivity}), the second being 
the absence of categorical ideals in the representation invariant under the $2$-action
(this is what we refer to as  {\em simplicity}).

For simple transitive $2$-representations we obtain, for arbitrary finitary $2$-categories, a
weak version of the classical Jordan-H{\"o}lder Theorem, see Theorem~\ref{thm4}, in which 
simple transitive $2$-representations appear as weak composition subquotients of general finitary
$2$-representations. It turns out that any finitary $2$-representation of a finitary $2$-category 
has a filtration with subquotients being transitive $2$-representations. In contrast to 
classical representation theory, transitive $2$-representations do not seem to admit any natural
filtration, however, they do have a well-defined simple top which is our weak composition subquotient.
A different approach to the Jordan-H{\"o}lder theory for $2$-Kac-Moody algebras is outlined in
\cite[Subsection~5.1]{Ro}.

Our main result is Theorem~\ref{thm15} which provides a classification of simple transitive $2$-representations
for a large class of finitary $2$-categories. The latter includes the $2$-category of Soergel bimodules in
type $A$, all of the above mentioned finitary quotients of $2$-Kac-Moody algebras and the $2$-category of 
projective functors on the module category of a finite dimensional self-injective algebra. Moreover, it
also includes all variations of the latter $2$-category which constitute a list of finitary $2$-categories
from \cite{MM3} satisfying a $2$-analogue of simplicity for a finite dimensional algebra. The classification
result states that for this class of $2$-categories simple transitive $2$-representations are precisely
the cell $2$-representations studied in \cite{MM1,MM2,MM3}. In particular, this implies uniqueness of
categorification of simple integrable modules for finite dimensional simple Lie algebras.
The only comparable statement in the literature,
for the $2$-categorical analogue of $U(\mathfrak{sl}_2)$ and for a special class of 
$2$-representations categorifying simple $\mathfrak{sl}_2$-modules, was proved in \cite[Proposition~5.26]{CR}.

The proof can be divided into two major parts. One of these (the proof of  Theorem~\ref{thm15}) reduces the 
problem to the case of the $2$-category of projective functors on the module category of a finite dimensional
self-injective algebra. The latter case is treated in Theorem~\ref{thmmain} and relies on a detailed study of
endomorphism algebras of certain bimodules and, crucially, on a classical result of Perron and Frobenius 
on the structure of real matrices with positive coefficients.

The article is organized as follows. In Section~\ref{s1} we recall notions developed in \cite{MM1,MM2,MM3}
and state the Perron-Frobenius Theorem. In Section~\ref{s2}  we introduce transitive and simple transitive
$2$-representations and gather examples and preliminary results. Section~\ref{s3} presents the statement and
proof of our weak Jordan-H{\"o}lder Theorem. Section~\ref{s4} is devoted to the proof of our main result
in the case of the $2$-category of projective functors on the module category of a finite dimensional
self-injective algebra. Section~\ref{s5} establishes the main result in the general case.
Finally, in Section~\ref{s6} we provide and study examples, including our family of quotients of 
$2$-Kac-Moody algebras.

\vspace{5mm}

\noindent
{\bf Acknowledgment.} A substantial part of the paper was written during mutual visits of the authors 
to the University of East Anglia respectively Uppsala University, whose hospitality is gratefully acknowledged. 
Both visits were supported by EPSRC grant EP/K011782/1. The first author is partially supported by the 
Swedish Research Council. The second author is partially supported by EPSRC grant EP/K011782/1. 
We thank Anne-Laure Thiel, Qimh Xantcha and Ben Webster for stimulating discussions. 
We thank the referee for very useful comments and explanations. 
We thank Ben Elias for pointing out a missing case in the
original version of the case-by-case analysis in Proposition~22.

\section{Preliminaries}\label{s1}

\subsection{Notation}\label{s1.1}

Throughout, we let $\Bbbk$ denote an algebraically closed field.  

A {\em $2$-category} is a category enriched over the category of small categories. A $2$-category $\cC$ consists of 
objects (denoted $\mathtt{i},\mathtt{j},\mathtt{k},\dots$), $1$-morphisms (denoted $\mathrm{F},\mathrm{G},\mathrm{H},
\dots$) and $2$-morphisms (denoted $\alpha,\beta,\gamma,\dots$). For $\mathtt{i}\in \cC$, the identity $1$-morphism
is denoted $\mathbbm{1}_{\mathtt{i}}$ and, for a $1$-morphism $\mathrm{F}$, the corresponding identity $2$-morphism 
is denoted  $\mathrm{id}_{\mathrm{F}}$. Composition of $1$-morphisms is denoted by $\circ$,
horizontal composition of $2$-morphisms is denoted by $\circ_0$ and vertical composition of $2$-morphisms is 
denoted by $\circ_1$. We let $\mathbf{Cat}$ denote the $2$-category of small categories.

\subsection{Finitary $2$-categories}\label{s1.2}

An additive $\Bbbk$-linear category is called {\em finitary} if it is idempotent split, has finitely many
isomorphism classes of indecomposable objects and finite dimensional $\Bbbk$-vec\-tor spaces of morphisms.
Denote by $\mathfrak{A}^f_{\Bbbk}$ the $2$-category whose objects are finitary additive $\Bbbk$-linear categories,
$1$-morphisms are additive $\Bbbk$-linear functors and $2$-morphisms are natural transformations of functors.

A {\em finitary} $2$-category (over $\Bbbk$) is a $2$-category $\cC$ with the following properties:
\begin{itemize}
\item it has a finite number of objects;
\item for any pair $\mathtt{i},\mathtt{j}$ of objects in $\cC$, the category 
$\cC(\mathtt{i},\mathtt{j})$ is in $\mathfrak{A}_{\Bbbk}^f$ 
and horizontal composition is both additive and $\Bbbk$-linear;
\item for any $\mathtt{i}\in\cC$, the $1$-morphism $\mathbbm{1}_{\mathtt{i}}$ is indecomposable.
\end{itemize}
We refer to \cite{Le,McL} for more general details on abstract $2$-categories and to
\cite{MM1,MM2,MM3,MM4} for more information on finitary $2$-categories.

\subsection{$2$-representations}\label{s1.3}

Let $\cC$ be a finitary $2$-category. By a {\em $2$-representation} of $\cC$ we mean a strict 
$2$-functor from $\cC$ to $\mathbf{Cat}$. By a {\em finitary $2$-representation} of $\cC$ we mean a strict 
$2$-functor from $\cC$ to $\mathfrak{A}_{\Bbbk}^f$. Our $2$-representations are generally denoted by 
$\mathbf{M},\mathbf{N},\dots$ with one exception: for $\mathtt{i}\in\cC$ we have the {\em principal} $2$-representation $\mathbb{P}_{\mathtt{i}}:=\cC(\mathtt{i},{}_-)$. Finitary $2$-representations of $\cC$ form a $2$-category,
denoted $\cC\text{-}\mathrm{afmod}$, whose $1$-morphisms are $2$-natural transformations and whose $2$-morphisms 
are modifications (see \cite{Le,MM3}).

Two $2$-representations $\mathbf{M}$ and $\mathbf{N}$ of $\cC$ are called {\em equivalent} if 
there exists a $2$-natural transformation $\Phi:\mathbf{M}\to\mathbf{N}$ such that $\Phi_{\mathtt{i}}$
is an equivalence for each $\mathtt{i}$.

Let $\mathbf{M}$ be a $2$-representation of $\cC$. Assume that $\mathbf{M}(\mathtt{i})$ is an idempotent split
additive category for each $\mathtt{i}\in\cC$. For any collection of objects 
$X_i\in \mathbf{M}(\mathtt{i}_i)$, where $i\in I$,
the additive closure of all objects of the form $\mathrm{F}X_i$, where $i\in I$ and $\mathrm{F}$ runs 
through all $1$-morphisms of $\cC$ is stable under the action of $\cC$ and hence inherits the structure of a
$2$-representation by restriction. This $2$-representation will be denoted $\mathbf{G}_{\mathbf{M}}(\{X_i:i\in I\})$.

To simplify notation, we will often write $\mathrm{F}\, X$ for $\mathbf{M}(\mathrm{F})\, X$ where
$\mathrm{F}$ is a $1$-morphism.

\subsection{Combinatorics of finitary $2$-categories}\label{s1.4}

Let $\cC$ be a finitary $2$-category. Denote by $\mathcal{S}(\cC)$ the multisemigroup of isomorphism 
classes of $1$-morphisms in $\cC$, see \cite[Section~3]{MM2}. As usual, we define the left preorder
$\geq_L$ on $\mathcal{S}(\cC)$ as follows: for two $1$-morphisms $\mathrm{F},\mathrm{G}$ we set
$\mathrm{G}\geq_L\mathrm{F}$ provided that there is a $1$-morphism $\mathrm{H}$ such that $\mathrm{G}$
is isomorphic to a direct summand of $\mathrm{H}\circ \mathrm{F}$. Equivalence classes for $\geq_L$
are called {\em left cells}. Right and two-sided preorders $\geq_R$ and $\geq_J$ and respective cells are
defined analogously.

\subsection{Weakly fiat and fiat $2$-categories}\label{s1.5}

For a $2$-category $\cC$ there are three ways of creating an opposite $2$-category.
\begin{itemize}
\item We can reverse both $1$- and $2$-morphisms. 
\item We can reverse only $1$-morphisms. 
\item We can reverse only $2$-morphisms. 
\end{itemize}
In the present paper we let $\cC^{\mathrm{op}}$ denote the first of the three choices above.

A finitary $2$-category $\cC$ is called {\em weakly fiat} provided that 
\begin{itemize}
\item there is a weak equivalence $*:\cC\to \cC^{\mathrm{op}}$;
\item for any pair $\mathtt{i}, \mathtt{j}\in\cC$ and every $1$-morphism
$\mathrm{F}\in\cC(\mathtt{i},\mathtt{j})$ we have
$2$-mor\-phisms $\alpha:\mathrm{F}\circ\mathrm{F}^*\to
\mathbbm{1}_{\mathtt{j}}$ and $\beta:\mathbbm{1}_{\mathtt{i}}\to
\mathrm{F}^*\circ\mathrm{F}$ such that 
$\alpha_{\mathrm{F}}\circ_1\mathrm{F}(\beta)=\mathrm{id}_{\mathrm{F}}$ and
$\mathrm{F}^*(\alpha)\circ_1\beta_{\mathrm{F}^*}=\mathrm{id}_{\mathrm{F}^*}$.
\end{itemize}
If $*$ is involutive, then $\cC$ is called {\em fiat}, see \cite{MM1,MM2}.

\subsection{$2$-ideals}\label{s1.6}

Let $\cC$ be a $2$-category. A {\em left $2$-ideal} $\cI$ of $\cC$ consists of the same objects as $\cC$
and for each pair $\mathtt{i},\mathtt{j}$ of objects an ideal $\cI(\mathtt{i},\mathtt{j})$ in
$\cC(\mathtt{i},\mathtt{j})$ such that $\cI$ is stable under the left horizontal multiplication with 
$1$- and $2$-morphisms in $\cC$. Similarly one defines {\em right $2$-ideals} and
{\em two-sided $2$-ideals}. The latter will simply be called {\em $2$-ideals}.
For example, each principal $2$-representation can be interpreted as a left $2$-ideal in $\cC$.

Let $\cC$ be a $2$-category and $\mathbf{M}$ be a $2$-representation of $\cC$. An {\em ideal}
$\mathbf{I}$ of $\mathbf{M}$ is a collection of ideals $\mathbf{I}(\mathtt{i})$ in $\mathbf{M}(\mathtt{i})$
for each $\mathtt{i}\in\cC$ stable under the action of $\cC$ in the following sense: for any morphism 
$\eta\in \mathbf{I}$ and any $1$-morphism $\mathrm{F}$ the composition
$\mathbf{M}(\mathrm{F})(\eta)$ (if it is defined) is in $\mathbf{I}$. For example, 
left $2$-ideals of $\cC$ give rise to ideals in principal $2$-representations.

\subsection{Abelianization}\label{s1.7}

Let $\mathcal{A}$ be a finitary additive $\Bbbk$-linear category. Then the
{\em abelianization} $\overline{\mathcal{A}}$ of $\mathcal{A}$ is the category whose objects are
diagrams $X\overset{\eta}{\longrightarrow}Y$ where $X,Y\in \mathcal{A}$ and $\eta\in\mathcal{A}(X,Y)$
and morphisms are equivalence classes of solid commutative diagrams of the form
\begin{displaymath}
\xymatrix{ 
X\ar[rr]^{\eta}\ar[d]_{\tau_1}&&Y\ar[d]^{\tau_2}\ar@{-->}[dll]_{\tau_3}\\
X'\ar[rr]_{\eta'}&&Y'
} 
\end{displaymath}
modulo the subspace spanned by those diagrams for which there exists $\tau_3$ as indicated by the dashed arrow
such that $\eta'\tau_3=\tau_2$. The category $\overline{\mathcal{A}}$ is abelian (cf. \cite{Fr}) 
and is equivalent to the category of modules over the finite dimensional $\Bbbk$-algebra
\begin{displaymath}
\mathrm{End}_{\mathcal{A}}(P)^{\mathrm{op}}\quad  \text{ where }\quad 
P:=\bigoplus_{Q\in\mathrm{Ind}(\mathcal{A})/\cong}Q.
\end{displaymath}

Let $\cC$ be a $2$-category and $\mathbf{M}$ a finitary $2$-representation of $\cC$. Then the
{\em abelianization} of $\mathbf{M}$ is the $2$-representation $\overline{\mathbf{M}}$ of $\cC$
which assigns to each $\mathtt{i}\in\cC$ the category $\overline{\mathbf{M}}(\mathtt{i})$ with
the action of $\cC$ defined on diagrams component-wise.

Directly from the definition it follows that the action of each $1$-morphism on the abelianization of
any finitary $2$-representation is right exact.

A finitary $2$-representation $\mathbf{M}$ of $\cC$ will be called {\em exact} provided that 
$\overline{\mathbf{M}}(\mathrm{F})$ is exact for any $1$-morphism $\mathrm{F}$ in $\cC$.
For example, any finitary $2$-representation of a weakly fiat $2$-category is exact.

\subsection{Perron-Frobenius Theorem}\label{s1.9}

We will use the following classical result due to Perron and Frobenius, see the original papers
\cite{Fro1,Fro2,Pe} or the detailed exposition in \cite[Chapter~8]{Me}.

\begin{theorem}\label{thm1}
Let $A=(a_{i,j})$ be a real $n\times n$ matrix with strictly positive coefficients. 
\begin{enumerate}[$($i$)$]
\item\label{thm1.1} $A$ has a positive real eigenvalue, call it $r$, such that any other (possibly complex)
eigenvalue of $A$ has a strictly smaller absolute value.
\item\label{thm1.2} The eigenvalue $r$ appears with multiplicity one in the characteristic polynomial of $A$.
\item\label{thm1.3} There exists a real eigenvector, call it $\mathbf{v}$, for eigenvalue $r$ with strictly 
positive coefficients, moreover, any real eigenvector of $A$ with strictly positive coefficients is a 
multiple of $\mathbf{v}$.
\item\label{thm1.4} The eigenvalue $r$ satisfies
\begin{displaymath}
\min_j\{\sum_ia_{ij}\}\leq r\leq \max_j\{\sum_ia_{ij}\}.
\end{displaymath}
\end{enumerate}
\end{theorem}

\begin{corollary}\label{corpf}
Assume that $A$ is as in Theorem~\ref{thm1} and has rank one. Then,
if either inequality in Theorem~\ref{thm1}\eqref{thm1.4} is an equality, then both inequalities are equalities
and all columns of $A$ coincide. 
\end{corollary}

\begin{proof}
If $A$ has rank one, then all columns of $A$ are proportional to $\mathbf{v}$ and the trace of $A$ equals $r$.
Assume, for example, that $\min_j\{\sum_ia_{ij}\}=\sum_ia_{i1}=r$.
Set $\lambda_1=1$ and for $j=2,3,\dots,n$ let $\lambda_j$
be the positive real number ($\geq 1$) such that the $j$-th column equals $\lambda_j$ times the first column.
Then, we have
\begin{displaymath}
\sum_ia_{i1}=r=\mathrm{trace}(A)=\sum_ia_{ii}=\sum_i\lambda_i a_{i1}\geq \sum_i a_{i1}=r.
\end{displaymath}
It follows that $\lambda_j=1$ for all $j$. The case where the second inequality is an equality is similar. 
\end{proof}

\section{Transitive $2$-representations}\label{s2}

In this section, $\cC$ will be a finitary $2$-category. 

\subsection{Definition}\label{s2.1}

Let $\mathbf{M}$ be a finitary $2$-representations of $\cC$. We will say that $\mathbf{M}$ is {\em transitive}
provided that for every $\mathtt{i}$ and for every non-zero object $X\in\mathbf{M}(\mathtt{i})$ we have
$\mathbf{G}_{\mathbf{M}}(\{X\})=\mathbf{M}$.

\subsection{Example: transitive group actions}\label{s2.2}

Let $G=(G,\cdot)$ be a finite group. Consider the finitary $2$-category $\cG=\cG_G$ defined as follows:
\begin{itemize}
\item $\cG$ has one object $\clubsuit$;
\item $1$-morphisms in $\cG$ are $\displaystyle \bigoplus_{g\in G} \mathrm{F}_{g}^{\oplus k_g}$ where all $k_g\geq 0$;
\item composition of $1$-morphisms is given by $\mathrm{F}_g\circ \mathrm{F}_h=\mathrm{F}_{gh}$ and extended by biadditivity;
\item non-zero $2$-morphisms between indecomposable $1$-morphisms are just scalar multiples of the identity,
$2$-morphisms between decomposable $1$-morphisms are matrices of morphisms between the corresponding
indecomposable summands;
\item vertical composition of $2$-morphisms is given by matrix multiplication;
\item horizontal composition of $2$-morphisms is given by tensor product of matrices.
\end{itemize}
The $2$-category $\cG$ is finitary by definition. Moreover, it is even a fiat $2$-category
(where $*$ is induced by  $g\mapsto g^{-1}$).

Let $H$ be a subgroup of $G$. Let $\mathcal{A}$ be a small category equivalent to $\Bbbk\text{-}\mathrm{mod}$.
Consider the category 
\begin{displaymath}
\mathcal{G}_{H,\mathcal{A}}:=\bigoplus_{gH\in G/H}\mathcal{A}_{(gH)}, 
\end{displaymath}
where $(gH)$ is a formal index. Now define the $2$-representation $\mathbf{M}_{H,\mathcal{A}}$ of $\cG$
\begin{itemize}
\item on the object by $\mathbf{M}_{H,\mathcal{A}}(\clubsuit)=\mathcal{G}_{H,\mathcal{A}}$;
\item on $1$-morphisms by $\mathbf{M}_{H,\mathcal{A}}(\mathrm{F}_g)=\left(\varphi_{xH,yH}\right)_{xH,yH\in G/H}$ where
\begin{displaymath}
\varphi_{xH,yH}=
\begin{cases}
\mathrm{Id}_{\mathcal{A}},&  gyH=xH;\\
0,&  \text{otherwise};\\ 
\end{cases}
\end{displaymath}
\item on $2$-morphisms $\mathbf{M}_{H,\mathcal{A}}$ in the obvious way using scalar multiples of the identity
natural transformations.
\end{itemize}
It follows from the definition that $\mathbf{M}_{H,\mathcal{A}}$ is a transitive $2$-representation of $\cG$.
This $2$-representation categorifies the classical transitive action of $G$ on $G/H$. 

Note that in the above construction instead of $\mathcal{A}$ we can take any small finitary additive $\Bbbk$-linear
category $\mathcal{B}$ with one isomorphism class of indecomposable objects. 

This example generalizes, in the obvious way, to finite semigroups. One major difference is that in the latter
case the $2$-category obtained will not be fiat but only finitary. Another difference is that while
any transitive action of a finite group on a finite set is equivalent to the action on some $G/H$, transitive
actions of semigroups are more complicated, see e.g. \cite[Chapter~10]{GM}.

\subsection{Cell $2$-representations}\label{s2.3}

Here we use the approach from \cite{MM2} to construct cell $2$-representations for arbitrary finitary $2$-categories. 

Let $\mathcal{L}$ be a left cell in $\cC$. Then there is 
$\mathtt{i}=\mathtt{i}_{\mathcal{L}}\in\cC$ such that every $1$-morphism in $\mathcal{L}$
has domain $\mathtt{i}$. Consider the principal $2$-representation $\mathbb{P}_{\mathtt{i}}$. 
For $\mathtt{j}\in\cC$ let $\mathbf{N}(\mathtt{j})$ denote the additive closure in 
$\mathbb{P}_{\mathtt{i}}(\mathtt{j})$ of all $1$-morphisms 
$\mathrm{F}\in \cC(\mathtt{i},\mathtt{j})$ such that $\mathrm{F}\geq_L \mathcal{L}$.
Then $\mathbf{N}$ is a $2$-subrepresentation of $\mathbb{P}_{\mathtt{i}}$. 

\begin{lemma}\label{lem2}
There is a unique maximal ideal $\mathbf{I}$ in $\mathbf{N}$ which does not contain 
$\mathrm{id}_{\mathrm{F}}$ for any $\mathrm{F}\in\mathcal{L}$.
\end{lemma}

\begin{proof}
Being an ideal of an additive category, $\mathbf{I}$ is uniquely determined by its morphisms between 
indecomposable objects. If $\mathrm{F}\in\mathcal{L}\cap\cC(\mathtt{i},\mathtt{j})$, 
then the algebra of $2$-endomorphisms of $\mathrm{F}$
is local as $\mathrm{F}$ is indecomposable. Therefore the part of 
$\mathrm{End}_{\ccC(\mathtt{i},\mathtt{j})}(\mathrm{F})$ contained
in $\mathbf{I}$ belongs to the radical of 
$\mathrm{End}_{\ccC(\mathtt{i},\mathtt{j})}(\mathrm{F})$. As the sum of two subspaces
of the radical is contained in the radical, we conclude that the sum of all left ideals in $\mathbf{N}$ 
which do not contain $\mathrm{id}_{\mathrm{F}}$ for any $\mathrm{F}\in\mathcal{L}$ still has the
latter property. The claim follows.
\end{proof}

The quotient $2$-functor $\mathbf{C}_{\mathcal{L}}:=\mathbf{N}/\mathbf{I}$, where $\mathbf{I}$ is given by
Lemma~\ref{lem2}, is called the {\em (additive) cell $2$-representation} of $\cC$ associated to $\mathcal{L}$.
From the definitions, it follows directly that $\mathbf{C}_{\mathcal{L}}$ is a transitive $2$-representation of
$\cC$.

\subsection{A more exotic example}\label{s2.4}

Similarly to Subsection~\ref{s2.2} one defines a $2$-category $\cC$ with one object, indecomposable
$1$-morphisms $\mathbbm{1}$ and $\mathrm{F}$, with the multiplication table
\begin{displaymath}
\begin{array}{c||c|c}
\circ&\mathbbm{1}&\mathrm{F}\\
\hline\hline
\mathbbm{1}&\mathbbm{1}&\mathrm{F}\\
\hline
\mathrm{F}&\mathrm{F}&\mathrm{F}\oplus\mathrm{F} 
\end{array}
\end{displaymath}
and only scalar multiples of the identity $2$-morphisms for indecomposable $1$-morphisms.
This $2$-category $\cC$ has two left cells (corresponding to the two indecomposable 
$1$-morphisms), so we have the respective cell $2$-representations. These are transitive, see Subsection~\ref{s2.3}.
Similarly to Subsection~\ref{s2.2} one can construct a rather different transitive $2$-representation on a
category $\mathcal{A}\oplus \mathcal{A}$, where $\mathcal{A}$ is as in Subsection~\ref{s2.2}, by mapping the
$1$-morphism $\mathrm{F}$ to the functor
\begin{displaymath}
\left(\begin{array}{cc}\mathrm{Id}_{\mathcal{A}} &\mathrm{Id}_{\mathcal{A}}\\
\mathrm{Id}_{\mathcal{A}}&\mathrm{Id}_{\mathcal{A}}\end{array}\right).
\end{displaymath}

\subsection{Simple transitive $2$-representations}\label{s2.5}

Let $\mathbf{M}$ be a transitive $2$-representation of $\cC$.

\begin{lemma}\label{lem3}
There is a unique maximal ideal $\mathbf{I}$ in $\mathbf{M}$ which does not contain any identity morphisms
apart from the one for the zero object.
\end{lemma}

\begin{proof}
Mutatis mutandis proof of Lemma~\ref{lem2}. 
\end{proof}

The main idea of the following definition generalizes \cite[Subsection~6.5]{MM2}. A
transitive $2$-representation $\mathbf{M}$  of $\cC$ is called {\em simple transitive} provided that 
its unique maximal ideal given by Lemma~\ref{lem3} is the zero ideal. 
For a transitive $2$-representation $\mathbf{M}$ denote by $\underline{\mathbf{M}}$ the quotient of
$\mathbf{M}$ by the ideal $\mathbf{I}$ given by Lemma~\ref{lem3}. We will loosely call $\underline{\mathbf{M}}$
the {\em simple transitive quotient} of $\mathbf{M}$.

\subsection{Examples of simple transitive $2$-representations}\label{s2.7}

Lemma~\ref{lem2} implies that each cell $2$-representation of $\cC$ is simple transitive.
Furthermore, transitive $2$-representations $\mathbf{M}_{H,\mathcal{A}}$ of $\cG$ constructed in Subsection~\ref{s2.2} 
are simple transitive (and these are not equivalent to cell $2$-representations in general). 
In fact, the next proposition shows that these exhaust all simple transitive $2$-representations of $\cG$.

\begin{proposition}\label{prop35}
Every simple transitive $2$-representations of $\cG$ is equivalent to $\mathbf{M}_{H,\mathcal{A}}$ 
for some subgroup $H$ of $G$ and a skeletal category $\mathcal{A}$ equivalent to $\Bbbk\text{-}\mathrm{mod}$.
\end{proposition}

\begin{proof}
Let $\mathbf{M}$ be a simple transitive $2$-representation of $\cG$. 
Invertibility of each $\mathrm{F}_g$ implies that $\mathrm{F}_g$ sends non-isomorphic objects to 
non-isomorphic objects, indecomposable objects to indecomposable objects 
and radical morphisms to radical morphisms. Therefore the ideal
$\mathbf{I}$ given by Lemma~\ref{lem3} coincides with the radical of $\mathbf{M}(\clubsuit)$.
By simple transitivity, we hence obtain that the radical of $\mathbf{M}(\clubsuit)$ is zero
and thus $\mathbf{M}(\clubsuit)$ is a semi-simple category.

As each $\mathrm{F}_g$ sends indecomposable objects to indecomposable objects, $G$ induces a transitive action on the set of isomorphism classes of indecomposable
objects in $\mathbf{M}(\clubsuit)$. Fix an indecomposable object $X\in\mathbf{M}(\clubsuit)$ and set
\begin{displaymath}
H:=\{h\in G\,:\, \mathrm{F}_h\, X\cong X\}. 
\end{displaymath}

Let $\mathcal{A}$ be a skeletal category equivalent to $\Bbbk\text{-}\mathrm{mod}$. 
Consider the (unique!) functor $\Phi:\mathbf{M}(\clubsuit)\to \mathcal{G}_H$ which sends an indecomposable object
$Y\cong \mathrm{F}_g\, X$ for some $g\in G$ to the unique indecomposable object in $\mathcal{A}_{(gH)}$. Then
$\Phi$ is easily checked to give an equivalence between $\mathbf{M}$ and $\mathbf{M}_{H,\mathcal{A}}$.
The claim follows.
\end{proof}

Note that Proposition~\ref{prop35} does not extend to all transitive $2$-representations in an obvious way.
For example, let $G$ be the cyclic group of order two. Then $G$ acts by automorphisms on the finite dimensional
$\Bbbk$-algebra $A$ given by the quiver
\begin{displaymath}
\xymatrix{ 
\mathtt{1}\ar@/^/[rr]^{a}&&\mathtt{2}\ar@/^/[ll]^{b}
}
\end{displaymath}
with relations $ab=ba=0$ (the non-trivial automorphism is given by the automorphism of the quiver which swaps
$\mathtt{1}$ with $\mathtt{2}$ and $a$ with $b$). This induces a transitive action of $G$ and hence of the
corresponding $2$-category $\cG$ on any skeletal category  equivalent to the category 
of finite dimensional projective $A$-modules. We refer to \cite[Section~2]{AM} for more details.

\subsection{Strongly simple $2$-representations are (simple) transitive}\label{s2.6}

In parallel to \cite[Subsection~6.2]{MM1}, we call a finitary $2$-representation $\mathbf{M}$ of $\cC$
{\em strongly simple} provided that for any $\mathtt{i},\mathtt{j}\in\cC$ with $\overline{\mathbf{M}}(\mathtt{i})$
nonzero, any simple object  $L\in\overline{\mathbf{M}}(\mathtt{i})$ and any pair $P,Q$ of indecomposable projectives 
in $\overline{\mathbf{M}}(\mathtt{j})$, there exist indecomposable $1$-morphisms $\mathrm{F}$ and $\mathrm{G}$
such that $\mathrm{F}\, L\cong P$, $\mathrm{G}\, L\cong Q$ and the evaluation map
$\mathrm{Hom}_{\ccC(\mathtt{i},\mathtt{j})}(\mathrm{F},\mathrm{G})\to 
\mathrm{Hom}_{\overline{\mathbf{M}}(\mathtt{j})}(\mathrm{F}\, L,\mathrm{G}\, L)$ is surjective.

\begin{proposition}\label{prop16}
Let $\cC$ be a finitary $2$-category and ${\mathbf{M}}$ a strongly simple finitary 
$2$-representation of $\cC$. 
\begin{enumerate}[$($i$)$]
\item\label{prop16.1} The $2$-representation ${\mathbf{M}}$ is transitive.
\item\label{prop16.2} 
If $\mathbf{M}$ is exact (in particular, if $\cC$ is weakly fiat), then $\mathbf{M}$ is simple transitive.
\end{enumerate}
\end{proposition}

\begin{proof}
Let $X$ be a non-zero indecomposable object in some ${\mathbf{M}}(\mathtt{i})$ and $L$ be its simple
top in $\overline{\mathbf{M}}(\mathtt{i})$. Let $Y$ be a non-zero indecomposable object in some 
${\mathbf{M}}(\mathtt{j})$. By definition of strong simplicity, there is an indecomposable 
$1$-morphism $\mathrm{F}$ such that $\mathrm{F}\, L\cong Y$. This means that $Y$ is isomorphic to a direct
summand of $\mathrm{F}\, X$ and hence $\mathbf{M}$ is transitive. This proves claim~\eqref{prop16.1}.

Let $X,Y\in \overline{\mathbf{M}}(\mathtt{i})$ be two indecomposable projective objects and
$\eta:X\to Y$ be a non-zero morphism. Denote by $L\in \overline{\mathbf{M}}(\mathtt{i})$ the simple
top of $X$. Choose two $1$-morphisms $\mathrm{F}$ and $\mathrm{G}$  in $\cC$ such that $\mathrm{F}\, L\cong X$
and $\mathrm{G}\, L\cong Y$. Consider a finite dimensional $\Bbbk$-algebra $B$ such that 
$\overline{\mathbf{M}}(\mathtt{i})\cong B\text{-}\mathrm{mod}$. For simplicity, we identify
$\overline{\mathbf{M}}(\mathtt{i})$ and $B\text{-}\mathrm{mod}$. Let $e,e'$ be two primitive idempotents
of $B$ such that $X\cong Be$ and $Y\cong Be'$. Then, by Lemma~\ref{lem8}, the functor
$\overline{\mathbf{M}}(\mathrm{F})$ surjects onto the projective functor $Be\otimes_{\Bbbk} eB\otimes_B{}_-$.
Similarly, the functor $\overline{\mathbf{M}}(\mathrm{G})$ surjects onto the projective 
functor $Be'\otimes_{\Bbbk} eB\otimes_B{}_-$.

Now, for any non-zero map $\eta':Be\to Be'$ the induced map
\begin{displaymath}
\mathrm{Id}_{Be}\otimes \mathrm{Id}_{eB}\otimes \eta': 
Be\otimes_{\Bbbk} eB\otimes_BBe\to Be\otimes_{\Bbbk} eB\otimes_BBe' 
\end{displaymath}
contains, as a direct summand, the identity map on $Be$. This implies that the ideal $\mathbf{I}$ in
$\mathbf{M}$ generated by $\eta$ contains the identity morphism on $X$. Therefore $\mathbf{M}$ is
simple transitive.
\end{proof}

\begin{example}\label{ex17}
{\rm 
The claim of Proposition~\ref{prop16}\eqref{prop16.2} fails for general finitary $2$-representations.
Consider the algebra $D=\Bbbk[x]/(x^2)$ of dual numbers. Let $\mathcal{A}$ be a small category 
equivalent to $D\text{-}\mathrm{mod}$ and $\cC$ the finitary category with one object $\clubsuit$ which we
identify with $\mathcal{A}$, with indecomposable $1$-morphisms being endofunctors of $\mathcal{A}$
isomorphic to either the identity functor or tensoring with the $D\text{-}D$-bimodule $D\otimes_{\Bbbk}\Bbbk$, and
$2$-morphisms being natural transformations of functors. Then the defining $2$-representation of $\cC$,
i.e. the natural $2$-action of $\cC$ on $\mathcal{A}$, is clearly strongly simple. However, as
tensoring with $D\otimes_{\Bbbk}\Bbbk$ annihilates the non-zero nilpotent endomorphism of ${}_D D$,
this $2$-representation is not simple transitive.
}
\end{example}

Note also that the example of a transitive $2$-representation considered in Subsection~\ref{s2.4} is, 
clearly, simple transitive but not strongly simple.

\section{Weak Jordan-H{\"o}lder theory}\label{s3}

In this section, $\cC$ will be a finitary $2$-category. 

\subsection{The action preorder}\label{s3.1}

Let $\mathbf{M}$ be a finitary $2$-representation of $\cC$. Consider the (finite) set $\mathrm{Ind}(\mathbf{M})$ of 
isomorphism classes of indecomposable objects in all $\mathbf{M}(\mathtt{i})$ where $\mathtt{i}\in\cC$. For
$X,Y\in \mathrm{Ind}(\mathbf{M})$ set $X\geq Y$ provided that there is a $1$-morphisms $\mathrm{F}$ in $\cC$
such that $X$ is isomorphic to a direct summand of $\mathrm{F}\, Y$. Clearly, $\geq$ is a partial preorder on
$\mathrm{Ind}(\mathbf{M})$ which we will call the {\em action preorder}. 

Let $\sim$ be the equivalence relation defined by $X\sim Y$ if and only if $X\geq Y$ and $Y\geq X$. 
Note that $\mathbf{M}$ is transitive  if and only if we have exactly one equivalence class, namely the whole of
$\mathrm{Ind}(\mathbf{M})$.  The preorder $\geq$ induces a genuine partial order on the set
$\mathrm{Ind}(\mathbf{M})/_{\sim}$ which, abusing notation, we will denote by the same symbol.

\subsection{$2$-subrepresentations and subquotients associated to coideals}\label{s3.2}

Let $Q$ be a coideal in $\mathrm{Ind}(\mathbf{M})/_{\sim}$. For $\mathtt{i}\in\cC$ consider the additive
closure $\mathbf{M}_Q(\mathtt{i})$ in $\mathbf{M}(\mathtt{i})$ of all indecomposable objects 
$X\in \mathbf{M}(\mathtt{i})$ whose equivalence class belongs to $Q$. Then $\mathbf{M}_Q$ has the natural
structure of a $2$-representations of $\cC$ given by restriction from $\mathbf{M}$. This is the
$2$-subrepresentation of $\mathbf{M}$ associated to $Q$.

Suppose we are given a pair $Q,R$ of coideals in $\mathrm{Ind}(\mathbf{M})/_{\sim}$ such that $Q\subset R$.
For $\mathtt{i}\in\cC$ let $\mathbf{I}(\mathtt{i})$ denote the ideal in $\mathbf{M}_{R}(\mathtt{i})$
generated by the identities on the objects in $\mathbf{M}_{Q}(\mathtt{i})$. Then we can form the
quotient category $\mathbf{M}_{R/Q}(\mathtt{i}):=\mathbf{M}_{R}(\mathtt{i})/\mathbf{I}(\mathtt{i})$
and the $2$-functor $\mathbf{M}_{R}$ induces the $2$-functor $\mathbf{M}_{R/Q}$ which sends
$\mathtt{i}$ to $\mathbf{M}_{R/Q}(\mathtt{i})$. This is the $2$-subquotient of $\mathbf{M}$ associated to $Q\subset R$.
Note that $|R\setminus Q|=1$ implies that the $2$-representation $\mathbf{M}_{R/Q}$ is transitive.

For $r\in \mathrm{Ind}(\mathbf{M})/_{\sim}$ let $X_r$ be the maximal coideal in $\mathrm{Ind}(\mathbf{M})/_{\sim}$
which does not contain $r$. Then $r$ becomes the minimum element in 
$\big(\mathrm{Ind}(\mathbf{M})/_{\sim}\big)\setminus X_r$
with  respect to the induced order. Let $Y_r:=X_r\cup\{r\}$. Then $Y_r$ is a coideal in
$\mathrm{Ind}(\mathbf{M})/_{\sim}$. Therefore we have the associated 
quotient $\mathbf{M}_{Y_r/X_r}$ and we set $\underline{\mathbf{M}}_r:=\underline{\mathbf{M}}_{Y_r/X_r}$.

\subsection{Weak Jordan-H{\"o}lder series}\label{s3.3}

Consider a filtration
\begin{displaymath}
\mathcal{Q}:\qquad
\varnothing=Q_0\subsetneq Q_1\subsetneq Q_2\subsetneq\dots\subsetneq Q_k=  \mathrm{Ind}(\mathbf{M})/_{\sim}
\end{displaymath}
of coideals  such that $|Q_i\setminus Q_{i-1}|=1$ for all $i$. Such a filtration will be called
a {\em complete filtration}. With such a filtration we associate a
filtration of $2$-subrepresentations
\begin{equation}\label{eq1}
0\subset \mathbf{M}_{Q_1}\subset \mathbf{M}_{Q_2}\subset \dots\subset \mathbf{M}_{Q_k}=\mathbf{M}
\end{equation}
and the corresponding sequence
\begin{equation}\label{eq2}
\underline{\mathbf{M}}_{Q_1}, \underline{\mathbf{M}}_{Q_2/Q_1},
\underline{\mathbf{M}}_{Q_3/Q_2},\dots, \underline{\mathbf{M}}_{Q_k/Q_{k-1}} 
\end{equation}
of {\em simple transitive subquotients}. The filtration \eqref{eq1} is called a {\em weak Jordan-H{\"o}lder series}
of $\mathbf{M}$ and the elements in \eqref{eq2} are also called {\em weak composition subquotients}.

\subsection{Weak Jordan-H{\"o}lder theorem}\label{s3.4}

The main result of this section is the following weak version of the classical Jordan-H{\"o}lder theorem.

\begin{theorem}\label{thm4}
Let $\cC$ be a finitary $2$-category and $\mathbf{M}$ a finitary $2$-representation of $\cC$.
Let further $\mathcal{Q}$ and $\mathcal{R}$ be two complete filtrations of $\mathrm{Ind}(\mathbf{M})/_{\sim}$.
Let $\mathbf{L}_1,\mathbf{L}_2,\dots,\mathbf{L}_k$ be the sequence of simple transitive subquotients associated
to $\mathcal{Q}$ and $\mathbf{L}'_1,\mathbf{L}'_2,\dots,\mathbf{L}'_l$ be the sequence of simple transitive 
subquotients associated to $\mathcal{R}$. Then $k=l$ and there is a bijection $\sigma:\{1,2,\dots,k\}\to
\{1,2,\dots,k\}$ such that $\mathbf{L}_i$ and $\mathbf{L}'_{\sigma(i)}$ are equivalent for all $i\in \{1,2,\dots,k\}$.
\end{theorem}

\begin{proof}
Note first that we have $k=l=|\mathrm{Ind}(\mathbf{M})/_{\sim}|$ by definition. Let 
$r\in \mathrm{Ind}(\mathbf{M})/_{\sim}$. Then there are unique $i,j\in \{1,2,\dots,k\}$ such that 
$r=Q_i\setminus Q_{i-1}$ and $r=R_j\setminus R_{j-1}$. To prove the assertion it is enough to show that the
$2$-representations $\underline{\mathbf{M}}_r$, $\mathbf{L}_i$ and $\mathbf{L}'_j$ are equivalent.
By symmetry, it is enough to show that $\underline{\mathbf{M}}_r$ and $\mathbf{L}_i$ are equivalent.

Let $\mathbf{I}$ be the ideal in $\mathbf{M}_{Y_r}$ used to define $\mathbf{M}_{Y_r/X_r}$.
Similarly, let $\mathbf{J}$ be the ideal in $\mathbf{M}_{Q_i}$ used to define $\mathbf{M}_{Q_i/Q_{i-1}}$.
By construction of $X_r$, we have $Q_{i-1}\subset X_r$ and hence also $Q_i\subset Y_r$. 
The inclusion $Q_i\subset Y_r$ induces a faithful  $2$-natural transformation 
from $\mathbf{M}_{Q_i}$ to $\mathbf{M}_{Y_r}$ which gives, by taking the quotient, a strong
transformation from $\mathbf{M}_{Q_i}$ to $\mathbf{M}_{Y_r/X_r}$. 
Since $Q_{i-1}\subset X_r$, for any indecomposable objects 
$M$ and $N$ whose $\sim$-classes belong to $r$, we have $\mathbf{J}(M,N)\subset \mathbf{I}(M,N)$.
Therefore the strong
transformation from $\mathbf{M}_{Q_i}$ to $\mathbf{M}_{Y_r/X_r}$ factors through
$\mathbf{M}_{Q_i/Q_{i-1}}$. This gives a $2$-natural transformation from 
$\mathbf{M}_{Q_i/Q_{i-1}}$ to $\mathbf{M}_{Y_r/X_r}$ which is surjective on morphisms. 
Note that both $2$-representations  $\mathbf{M}_{Q_i/Q_{i-1}}$ and  $\mathbf{M}_{Y_r/X_r}$
are transitive. Taking now the quotient by the unique maximal ideal given by Lemma~\ref{lem3} induces
an equivalence between the corresponding simple transitive quotients, that is between
$\mathbf{L}_i$ and $\underline{\mathbf{M}}_r$. The claim follows.
\end{proof}

\subsection{Example: weak composition subquotients for principal $2$-rep\-re\-sen\-ta\-ti\-ons}\label{s3.5}

Consider the principal $2$-representation $\mathbb{P}_{\mathtt{i}}$ for $\mathtt{i}\in\cC$. The 
action preorder $\geq$ for $\mathbb{P}_{\mathtt{i}}$ coincides with the restriction to $\mathbb{P}_{\mathtt{i}}$
of the preorder $\geq_L$. Therefore $\mathrm{Ind}(\mathbb{P}_{\mathtt{i}})$ coincides with the set of isomorphism
classes of $1$-morphisms in $\cC$ with domain $\mathtt{i}$. The set 
$\mathrm{Ind}(\mathbb{P}_{\mathtt{i}})/_{\sim}$ thus becomes the set of all left cells with domain $\mathtt{i}$.
Comparing Subsection~\ref{s2.3} with Subsection~\ref{s3.3}, we see that weak composition subquotients
of $\mathbb{P}_{\mathtt{i}}$ are exactly the cell $2$-representations for left cells with domain $\mathtt{i}$.

\section{Classification of transitive $2$-representations for $\cC_{A}$}\label{s4}

\subsection{The $2$-category $\cC_{A}$}\label{s4.1}

Let $A$ be a basic self-injective connected $\Bbbk$-algebra of finite dimension $m$. Fix a small category $\mathcal{A}$
equivalent to $A\text{-}\mathrm{mod}$. We assume that $\mathcal{A}$ is not semi-simple.
Define the $2$-category $\cC_{A}$ as follows (cf.  \cite[Subsection~7.3]{MM1}):
\begin{itemize}
\item $\cC_{A}$ has one object $\clubsuit$ (which we identify with $\mathcal{A}$);
\item $1$-morphisms in $\cC_{A}$ are direct sums of functors with summands isomorphic to the identity
functor or to tensoring with projective $A\text{-}A$-bimodules; 
\item $2$-morphisms in $\cC_{A}$ are natural transformations of functors.
\end{itemize}
Functors isomorphic to tensoring with projective $A\text{-}A$-bimodules will be called {\em projective functors}.

Fix some decomposition $1=e_1+e_2+\dots+e_n$ of the identity in $A$ into a sum of primitive orthogonal idempotents.
The $2$-category $\cC_{A}$ has a unique minimal two-sided cell consisting of  the isomorphism class of the identity
morphism. It has one other two-sided cell $\mathcal{J}$ consisting of the isomorphism classes
of functors $\mathrm{F}_{ij}$ given by tensoring with the indecomposable bimodules $Ae_i\otimes e_jA$, 
where $i,j\in\{1,2,\dots,n\}$. Left and right cells in $\mathcal{J}$ are 
\begin{displaymath}
\mathcal{L}_j:=\{\mathrm{F}_{ij}:i\in\{1,2,\dots,n\}\}\quad\text{ and }\quad
\mathcal{R}_i:=\{\mathrm{F}_{ij}:j\in\{1,2,\dots,n\}\},
\end{displaymath}
where $i,j\in\{1,2,\dots,n\}$. We have 
\begin{displaymath}
\mathrm{F}_{ij}\circ \mathrm{F}_{st}\cong  \mathrm{F}_{it}^{\oplus\dim(e_jAe_s)}.
\end{displaymath}
Let $\sigma:\{1,2,\dots,n\}\to \{1,2,\dots,n\}$ be the {\em Nakayama} bijection given by requiring
$\mathrm{soc}\,Ae_i\cong \mathrm{top}\,Ae_{\sigma(i)}$ which is equivalent to 
$Ae_i\cong \mathrm{Hom}_{\Bbbk}(e_{\sigma(i)}A,\Bbbk)$. Since
\begin{displaymath}
\mathrm{Hom}_{A}(Ae_i\otimes_{\Bbbk}e_{j}A,{}_-)\cong
\mathrm{Hom}_{\Bbbk}(e_{j}A,\Bbbk)\otimes_{\Bbbk} e_iA\otimes_A {}_-,
\end{displaymath}
see  e.g. \cite[Subsection~7.3]{MM1}, we have that $(\mathrm{F}_{ij},\mathrm{F}_{\sigma^{-1}(j)i})$ 
is an adjoint pair of functors. This implies that $\cC_{A}$ is weakly fiat with $*$ defined on $1$-morphisms by
$\mathrm{F}_{ij}^*=\mathrm{F}_{\sigma^{-1}(j)i}$.

We set $\displaystyle\mathrm{F}:=\bigoplus_{i,j=1}^n\mathrm{F}_{ij}$. Since $A$ is basic and
\begin{displaymath}
A\otimes_{\Bbbk}A\otimes_A A\otimes_{\Bbbk}A\cong A\otimes_{\Bbbk}A^{\oplus m}, 
\end{displaymath}
we have 
\begin{equation}\label{eq5}
\mathrm{F}\circ \mathrm{F}\cong \mathrm{F}^{\oplus m}. 
\end{equation}
Note that $\mathrm{F}^*\cong \mathrm{F}$.

The $2$-category $\cC_A$ is {\em $\mathcal{J}$-simple} in the sense that any nonzero two-sided
$2$-ideal in $\cC_A$ contains the identity $2$-morphisms on all indecomposable non-identity $1$-morphisms,
see \cite[Subsection~6.2]{MM2}.

Denote by $\mathcal{P}$ the full subcategory of $\mathcal{A}$ consisting of projective
objects. Then the defining action of $\cC_A$ on $\mathcal{A}$ restricts to $\mathcal{P}$.
We will denote the latter {\em defining additive} $2$-representation of $\cC_A$ by $\mathbf{D}$.

\begin{proposition}\label{prop11}
For any $j=1,\dots,n$ the $2$-representations $\mathbf{D}$ and $\mathbf{C}_{\mathcal{L}_j}$ are equivalent.
\end{proposition}

\begin{proof}
It is easy to check that mapping the generator $P_{\mathbbm{1}_{\clubsuit}}$ of $\mathbb{P}_{\clubsuit}$
to the simple object in $\mathcal{A}$ corresponding to $j$ induces an equivalence 
from $\mathbf{C}_{\mathcal{L}_j}$ to $\mathbf{D}$.
\end{proof}

\subsection{Matrices in the Grothendieck group}\label{s4.2}

Let $\mathbf{M}$ be a finitary $2$-representation of $\cC_{A}$. For a $1$-morphism $\mathrm{G}$ denote
by $[\mathrm{G}]$ the square matrix with non-negative integer coefficients whose rows and columns are
indexed by isomorphism classes of indecomposable objects in $\mathbf{M}(\clubsuit)$ and the intersection
of the row indexed by $Y$ and the column indexed by $X$ contains the multiplicity of $Y$ as a direct
summand of $\mathrm{G}\, X$.

Consider the abelianization $\overline{\mathbf{M}}$ of $\mathbf{M}$. Then the isomorphism classes of simple
objects in $\overline{\mathbf{M}}(\clubsuit)$ are in bijection with isomorphism classes of indecomposable 
objects in $\mathbf{M}(\clubsuit)$. For a $1$-morphism $\mathrm{G}$ denote
by $\llbracket\mathrm{G}\rrbracket$ the square matrix with non-negative integer coefficients whose rows and columns are
indexed by isomorphism classes of simple objects in $\overline{\mathbf{M}}(\clubsuit)$ 
and the intersection of the row indexed by $Y$ and the column indexed by $X$ contains the composition
multiplicity of $Y$ in $\mathrm{G}\, X$. The following generalizes \cite[Lemma~8]{AM}.

\begin{lemma}\label{lem6}
We have $\llbracket\mathrm{G}^*\rrbracket=[\mathrm{G}]^t$, where ${{}_-}^t$ 
denotes the transpose of a matrix. 
\end{lemma}

\begin{proof}
For a projective $P$ and a simple $L$ in  $\overline{\mathbf{M}}(\clubsuit)$ we have
\begin{displaymath}
\mathrm{Hom}_{\overline{\mathbf{M}}(\clubsuit)}(\mathrm{G}\,P,L)\cong  
\mathrm{Hom}_{\overline{\mathbf{M}}(\clubsuit)}(P,\mathrm{G}^*\,L).
\end{displaymath}
The inclusion of $\mathbf{M}(\clubsuit)$ to $\overline{\mathbf{M}}(\clubsuit)$ given by $X\mapsto (0\to X)$
is an equivalence between $\mathbf{M}(\clubsuit)$ and the category of projective objects in 
$\overline{\mathbf{M}}(\clubsuit)$. This implies the claim.
\end{proof}

\begin{lemma}\label{lem5}
Consider the functor $\mathrm{F}$ from Subsection~\ref{s4.1}.
\begin{enumerate}[$($i$)$]
\item\label{lem5.1} The matrix $[\mathrm{F}]$ satisfies $[\mathrm{F}]^2=m[\mathrm{F}]$.
\item\label{lem5.2} If $\mathbf{M}$ is transitive, then all entries in $[\mathrm{F}]$ are positive.
\item\label{lem5.3} If $\mathbf{M}$ is transitive, then the rank of $[\mathrm{F}]$ equals one.
\end{enumerate}
\end{lemma}

\begin{proof}
Claim~\eqref{lem5.1} follows from \eqref{eq5}. Claim~\eqref{lem5.2} is immediate from the definition of transitivity.

Claim~\eqref{lem5.1} implies that $[\mathrm{F}]$ is diagonalizable with eigenvalues $0$ and $m$.
By Theorem~\ref{thm1}\eqref{thm1.2}, the eigenvalue $m$ has multiplicity one. Claim~\eqref{lem5.3} follows.
\end{proof}

\subsection{Auxiliary lemmata}\label{s4.3}

\begin{lemma}\label{lem7}
Let $\mathbf{M}$ be a simple transitive $2$-representation of $\cC_{A}$. Then for any
$X\in \overline{\mathbf{M}}(\clubsuit)$ the object $\mathrm{F}\, X$ is projective in 
$\overline{\mathbf{M}}(\clubsuit)$.
\end{lemma}

\begin{proof}
Applying $\mathrm{F}$  to a minimal projective presentation $P_1\overset{\alpha}{\longrightarrow} P_0$ 
of $\mathrm{F}\, X$ we get a projective  presentation 
$\mathrm{F}\, P_1\overset{\mathrm{F}(\alpha)}{\longrightarrow}  \mathrm{F}\, P_0$ of
$\mathrm{F}^2\, X\cong (\mathrm{F}\, X)^{\oplus m}$.

Consider the split Grothendieck group of the category $\mathcal{W}$ 
of projective objects in $\overline{\mathbf{M}}(\clubsuit)$.
For $i=0,1$ let $v_i$ be the vector recording the multiplicities of indecomposable projective objects in
$\mathrm{F}\, P_i$. Then, by minimality of the presentation $P_1\overset{\alpha}{\longrightarrow} P_0$, we have
\begin{equation}\label{eq3}
[\mathrm{F}]\cdot v_i=mv_i+w_i 
\end{equation}
for some non-negative vectors $w_i$. Note that $mv_i+w_i$ is a nonzero vector and belongs to the image of $[\mathrm{F}]$.
Therefore, by Lemma~\ref{lem5}\eqref{lem5.3}, $mv_i+w_i$ is an eigenvector for $[\mathrm{F}]$ with eigenvalue $m$.
Hence $[\mathrm{F}](mv_i+w_i)=m(mv_i+w_i)$. On the other hand,
\begin{displaymath}
[\mathrm{F}](mv_i+w_i)=m[\mathrm{F}]v_i+[\mathrm{F}]w_i\overset{\eqref{eq3}}{=}m(mv_i+w_i)+ [\mathrm{F}]w_i.
\end{displaymath}
Therefore $[\mathrm{F}]w_i=0$ and since $w_i$ has only non-negative entries and all entries of $[\mathrm{F}]$ are
positive, we obtain $w_i=0$.

It follows that $\mathrm{F}\, P_1\overset{\mathrm{F}(\alpha)}{\longrightarrow}  \mathrm{F}\, P_0$ is a
minimal projective presentation of $\mathrm{F}^2\, X$, in particular, the morphism $\mathrm{F}(\alpha)$ 
is contained in the radical of $\overline{\mathbf{M}}(\clubsuit)$. 

The category $\mathcal{W}$ carries the structure of a
$2$-rep\-re\-sen\-ta\-tion of $\cC_{A}$ by restriction. This $2$-representation is equivalent to $\mathbf{M}$
(the natural inclusion of $\mathbf{M}(\clubsuit)$ into $\mathcal{W}$ is the desired equivalence).
In particular, the $2$-rep\-re\-sen\-ta\-tion of $\cC_{A}$ on $\mathcal{W}$ is simple transitive.
Let $\mathcal{I}$ be the ideal of $\mathcal{W}$ generated by $\mathrm{F}(\alpha)$. This is contained in the
radical of $\mathcal{W}$ by the above and is $\mathrm{F}$-stable by \eqref{eq5}. 
Hence $\mathcal{I}$ is $\cC_{A}$-stable as it is stable under all indecomposable non-identity $1$-morphisms.
By simple transitivity, we thus get $\mathcal{I}=0$, that is $\alpha=0$. The claim follows.
\end{proof}

\begin{lemma}\label{lem8}
Let $B$ be a finite dimensional $\Bbbk$-algebra and $\mathrm{G}$ an exact endofunctor of $B\text{-}\mathrm{mod}$.
Assume that $\mathrm{G}$ sends each simple object of $B\text{-}\mathrm{mod}$  to a projective object.
Then $\mathrm{G}$ is a projective functor.
\end{lemma}

\begin{proof}
Consider a short exact sequence of functors  $\mathrm{K}\hookrightarrow\mathrm{H}\tto \mathrm{G}$ where
$\mathrm{H}$ is a projective functor. This exists because any right exact functor is equivalent to tensoring with
some bimodule and is hence a quotient of a projective functor. We assume that $\mathrm{H}$ is chosen minimally,
that is such that the tops of $\mathrm{H}$ and $\mathrm{G}$ (viewed as bimodules) agree.

Applying $\mathrm{K}\hookrightarrow\mathrm{H}\tto \mathrm{G}$ to a short exact sequence
$X\hookrightarrow Y\tto Z$ in $B\text{-}\mathrm{mod}$ we observe that $\mathrm{H}\,X\tto \mathrm{G}\,X$ and hence
the Snake Lemma yields the exact sequence $\mathrm{K}\,X\hookrightarrow \mathrm{K}\,Y\tto \mathrm{K}\,Z$.
This implies that $\mathrm{K}$ is exact.

Applying $\mathrm{K}\hookrightarrow\mathrm{H}\tto \mathrm{G}$ to a simple object $L\in B\text{-}\mathrm{mod}$ we obtain
an exact sequence $\mathrm{K}\,L\hookrightarrow\mathrm{H}\,L\tto \mathrm{G}\,L$. By our choice of $\mathrm{H}$,
we have $\mathrm{H}\,L=0$ if and only if $\mathrm{G}\,L=0$. Furthermore, by assumptions on $\mathrm{G}$ we have
$\mathrm{H}\,L\cong \mathrm{G}\,L$ whenever $\mathrm{G}\,L\neq 0$. This implies 
$\mathrm{H}\,L\cong \mathrm{G}\,L$ for all $L$ and hence $\mathrm{K}\,L=0$. By exactness of $\mathrm{K}$ we thus
deduce $\mathrm{K}=0$ and hence $\mathrm{H}\cong \mathrm{G}$.
\end{proof}

\begin{lemma}\label{lem9}
Let $A$, $\cC_A$ and $\mathrm{F}$ be  as given in Subsection~\ref{s4.1}. Let further $\mathbf{M}$ be a
$2$-representation of $\cC_A$ and $N\in\overline{\mathbf{M}}(\clubsuit)$ such that 
$\mathrm{F}\, N\neq 0$. Then there is an algebra monomorphism from $A$ to 
$\mathrm{End}_{\overline{\mathbf{M}}(\clubsuit)}(\mathrm{F}\, N)$.
\end{lemma}

\begin{proof}
From the definitions we know that the $2$-endomorphism algebra of $\mathrm{F}$ is isomorphic to 
$A\otimes_{\Bbbk}A^{\mathrm{op}}$. We have a natural algebra monomorphism from $A$ to $A\otimes_{\Bbbk}A^{\mathrm{op}}$
given by $a\mapsto a\otimes 1$. Consider the evaluation homomorphism
\begin{displaymath}
\xymatrix{
\mathrm{End}_{\ccC(\clubsuit,\clubsuit)}(\mathrm{F}) 
\ar[rr]^{\mathrm{ev}_N}&&\mathrm{End}_{\overline{\mathbf{M}}(\clubsuit)}(\mathrm{F}\, N).
}
\end{displaymath}

For a fixed left cell $\mathcal{L}$ consider the corresponding cell $2$-representation $\mathbf{C}_{\mathcal{L}}$
of $\cC_A$. By \cite[Proposition~21]{MM2}, there is a unique maximal left ideal in $\cC_A$
which does not contain any identity $2$-morphisms for $1$-morphisms in $\mathcal{L}$. Now, by
\cite[Subsection~6.5]{MM2}, this left ideal is the annihilator of the sum of all simple objects in 
$\overline{\mathbf{C}}_{\mathcal{L}}$. From Proposition~\ref{prop11} we know that $\mathbf{C}_{\mathcal{L}}$
is equivalent to the defining representation which implies that this maximal left ideal is, in fact,
$A\otimes\mathrm{rad}\,A^{\mathrm{op}}$. Therefore the kernel of $\mathrm{ev}_N$, which is a left ideal, 
must belong to $A\otimes\mathrm{rad}\,A^{\mathrm{op}}$.

This implies that the kernel of $\mathrm{ev}_N$ does not intersect the space $A\otimes 1$ and hence
the induced composition $A\to  \mathrm{End}_{\ccC(\clubsuit,\clubsuit)}(\mathrm{F})
\to \mathrm{End}_{\overline{\mathbf{M}}(\clubsuit)}(\mathrm{F}\, N)$ is injective. 
\end{proof}

\subsection{Main result}\label{s4.4}

\begin{theorem}\label{thmmain}
Let $A$ be as given in Subsection~\ref{s4.1}. Then any simple transitive 
$2$-rep\-re\-sen\-ta\-tion of $\cC_{A}$ is equivalent to some cell $2$-representation.
\end{theorem}

\begin{proof}
Consider a simple transitive $2$-representation $\mathbf{M}$ of $\cC_{A}$ and its abelianization $\overline{\mathbf{M}}$.
Let $X_1,X_2,\dots,X_k$ be a complete and irredundant list of representatives of isomorphism classes of 
indecomposable objects in $\mathbf{M}(\clubsuit)$. Denote by $B$ the endomorphism algebra of
$\displaystyle\bigoplus_{i=1}^k X_i$. Note that $\overline{\mathbf{M}}(\clubsuit)$ is equivalent to
$B^{\mathrm{op}}\text{-}\mathrm{mod}$. For $i=1,2,\dots,k$ we let $L_i$ denote the simple quotient in 
$\overline{\mathbf{M}}(\clubsuit)$ of the indecomposable projective object $0\to X_i$.

Recall the $1$-morphism $\mathrm{F}$ defined in Subsection~\ref{s4.1} and the corresponding matrix 
$\llbracket \mathrm{F}\rrbracket$ describing the action of $\mathrm{F}$ on the Grothendieck group of
$\overline{\mathbf{M}}(\clubsuit)$ in the basis of simple modules.
By Theorem~\ref{thm1}\eqref{thm1.4}, there is a column,  
\begin{displaymath}
\left(\begin{array}{c}v_1\\v_2\\ \vdots\\v_k\end{array}\right)
\end{displaymath}
in $\llbracket \mathrm{F}\rrbracket$, say with index $j$, such that $v_1+v_2+\dots +v_k\leq m$. 
By Lemma~\ref{lem7} we have
\begin{displaymath}
\mathrm{F}\, L_j\cong \bigoplus_{i=1}^k X_i^{\oplus l_i} 
\end{displaymath}
for some non-negative integers $l_1,l_2,\dots,l_k$. Transitivity of $\mathbf{M}$ and \eqref{eq5} imply
that all $l_1,l_2,\dots,l_k$ are, in fact, positive integers. Denote by $B'$ the endomorphism algebra 
of $\mathrm{F}\, L_j$ which is Morita equivalent to $B$ by the previous sentence. The vector $(l_1,l_2,\dots,l_k)^t$
is, by \eqref{eq5}, an eigenvector of $[\mathrm{F}]$. Moreover, by Lemma~\ref{lem6} we  have
$[\mathrm{F}]=\llbracket \mathrm{F}^*\rrbracket^t=\llbracket \mathrm{F}\rrbracket^t$ where the latter
equality follows from self-adjointness of $\mathrm{F}$.

Lemma~\ref{lem9} provides an algebra embedding of $A$ into $B'$ and hence an embedding
$A_A\hookrightarrow B'_A$ of $A$-modules. Since the algebra $A$ (and hence also $A^{\mathrm{op}}$) is self-injective,
each indecomposable summand of $A_A$ has simple socle. Therefore the embedding $A_A\hookrightarrow B'_A$
induces an embedding (of right $A$-modules) from $A_A$ into 
\begin{equation}\label{eq6}
\bigoplus_{i=1}^k\mathrm{Hom}_{\overline{\mathbf{M}}(\clubsuit)}(X_i,\mathrm{F}\, L_j). 
\end{equation}
The dimension of the latter equals $v_1+v_2+\dots+v_k\leq m$, while $\dim_{\Bbbk} A=m$, therefore
$v_1+v_2+\dots+v_k=m$ and by Corollary~\ref{corpf} all columns of $\llbracket \mathrm{F}\rrbracket$
coincide. In particular, it follows that $l_1=l_2=\dots=l_k=l$ for some $l\in\mathbb{N}$ and thus
$B'$ is isomorphic to the algebra of $l\times l$ matrices with coefficients in $B$.

The algebra of $B'$-endomorphisms of \eqref{eq6} is isomorphic to $B$ and embeds into the algebra of
$A$-endomorphisms of \eqref{eq6} (the latter embedding is due to the fact that 
$A$ is a subalgebra of $B'$) which is equal to $A$ by comparing  dimensions. Therefore we have
\begin{displaymath}
B\hookrightarrow A\hookrightarrow B'. 
\end{displaymath}

Next we argue that $\mathrm{F}\, L_s=(X_1\oplus X_2\oplus \dots \oplus X_k)^{\oplus l}$ for any $s$.
The arguments above imply that $\mathrm{F}\, L_s=(X_1\oplus X_2\oplus \dots \oplus X_k)^{\oplus l_s}$
for some positive integer $l_s$. Now $l=l_s$ since all columns of
$\llbracket \mathrm{F}\rrbracket$ are equal.

As $\mathrm{F}\, L_j=(X_1\oplus X_2\oplus \dots \oplus X_k)^{\oplus l}$, it follows that 
$\dim_{\Bbbk}B'=lm$ and therefore $\dim_{\Bbbk}B=\frac{m}{l}$. Set 
$\Theta:=\overline{\mathbf{M}}(\mathrm{F})$. Lemma~\ref{lem8} implies that $\Theta$ is a projective functor
which sends each simple to $(X_1\oplus X_2\oplus \dots \oplus X_k)^{\oplus l}$.
The dimension of the endomorphism algebra of $\Theta$ thus equals $l\cdot l\cdot \frac{m}{l}\cdot \frac{m}{l}=m^2$.
Note that $\mathcal{J}$-simplicity of $\cC_A$ gives us a natural inclusion of the algebra 
$\mathrm{End}_{\ccC(\clubsuit,\clubsuit)}(\mathrm{F})\cong A\otimes A^{\mathrm{op}}$ of $2$-endomorphisms of $\mathrm{F}$
into the endomorphism algebra of $\Theta$ in the category of right exact endofunctors of
$\overline{\mathbf{M}}(\clubsuit)$. 
As both these algebras have dimension $m^2$, this natural inclusion is, in fact, an isomorphism.

Therefore $B\cong A\cong B'$ and thus $\overline{\mathbf{M}}$ is equivalent to the defining 
$2$-representation of $\cC_A$. Now the proof is completed by applying Proposition~\ref{prop11}. 
\end{proof}

\subsection{Generalizations}\label{s4.5}

\begin{remark}\label{rem12}
{\rm 
Theorem~\ref{thmmain} generalizes verbatim and with the same proof to the case where $A$ is a 
basic self-injective finite dimensional $\Bbbk$-algebra (not necessarily connected). The technical
difficulty in this case is that, in order to be consistent with the requirement for $\mathbbm{1}_{\mathtt{i}}$
to be indecomposable, one has to consider a $2$-category with several objects indexed by connected
components of $A$. 
}
\end{remark}

\begin{remark}\label{rem14}
{\rm  
Theorem~\ref{thmmain} generalizes verbatim to $2$-subcategories of $\cC_A$ described in \cite[Subsection~4.5]{MM3}.
These $2$-subcategories exhaust all ``simple'' $2$-categories of certain type, see \cite[Theorem~13]{MM3} 
and Subsection~\ref{s5.1} below for details.
The only difference between those $2$-subcategories and $\cC_A$ is that the former may contain fewer
$2$-endomorphisms of the identity $1$-morphisms. We did not use $2$-endomorphisms of identity $1$-morphisms
in the above proof.
}
\end{remark}

\section{Transitive $2$-representations for some general fiat $2$-categories}\label{s5}

\subsection{Strong regularity and a numerical condition}\label{s5.1}

Let $\cC$ be a fiat $2$-category and $\mathcal{J}$ a two-sided cell in $\cC$. We say that $\mathcal{J}$ is 
{\em strongly regular}, see \cite[Subsection~4.8]{MM1}, provided that 
\begin{itemize}
\item different right (left) cells in $\mathcal{J}$ are not comparable with respect to the right
(left) preorder;
\item the intersection of a left and a right cell in $\mathcal{J}$ consists of exactly one isomorphism
class of indecomposable $1$-morphisms.
\end{itemize}
For example, the $2$-category $\cC_A$ from Subsection~\ref{s4.1} is strongly regular.

If $\mathcal{J}$ is strongly regular, we have a well-defined function sending $\mathrm{F}\in \mathcal{J}$ to 
the number of indecomposable summands in $\mathrm{F}^*\circ\mathrm{F}$ which belong to $\mathcal{J}$.
We will say that $\mathcal{J}$ satisfies the {\em numerical condition} provided that this function
is constant on right cells. Again, it is easy to check that the $2$-category $\cC_A$ 
from Subsection~\ref{s4.1} satisfies the {numerical condition},
see \cite[Subsection~7.3]{MM1}.

Another example of a $2$-category in which each two-sided cell is strongly regular and
satisfies the numerical condition is the $2$-category $\cS_n$ of Soergel bimodules for the symmetric group $S_n$, 
see \cite[Subsection~7.1]{MM1} and \cite[Example~3]{MM2} for details.

\subsection{Another generalization of the main result}\label{s5.2}

\begin{theorem}\label{thm15}
Let $\cC$ be a fiat $2$-category such that all two-sided cells in $\cC$ are strongly regular and
satisfy the numerical condition. Then any simple transitive $2$-representation of $\cC$ is equivalent
to a cell $2$-representation.
\end{theorem}

\begin{proof}
Let $\mathbf{M}$ be a simple transitive $2$-representation of $\cC$. First of all, we claim that there is
a unique maximal two-sided cell $\mathcal{J}$ which does not annihilate $\mathbf{M}$. Indeed, assume that
we have two maximal two-sided cells $\mathcal{J}_i$ for $i=1,2$ with this property. Then for any 
$\mathrm{F}_i\in \mathcal{J}_i$, $i=1,2$, we have $\mathbf{M}(\mathrm{F}_1)\circ \mathbf{M}(\mathrm{F}_2)=0$
and $\mathbf{M}(\mathrm{F}_2)\circ \mathbf{M}(\mathrm{F}_1)=0$ whenever the expression makes sense.
Therefore the additive closure of objects in all $\mathbf{M}(\mathtt{i})$ which may be obtained by applying 
$1$-morphisms from $\mathcal{J}_1$ is, on the one hand, a $2$-subrepresentation of $\mathbf{M}$ (by maximality
of $\mathcal{J}_1$) and, on the other hand, annihilated by all $1$-morphisms from $\mathcal{J}_2$.
Due to transitivity of $\mathbf{M}$, we obtain that $\mathcal{J}_2$ annihilates $\mathbf{M}$, a contradiction.
 
Now denote by $\mathcal{J}$ the maximal two-sided cell of $\cC$ which does not annihilate $\mathbf{M}$.
Without loss of generality we may assume that $\mathcal{J}$ is the unique maximal two-sided cell in $\cC$
and that $\mathbf{M}$ is $2$-faithful in the sense that it does not annihilate any $2$-morphisms. Indeed,
we may replace $\cC$ by its quotient modulo the kernel of $\mathbf{M}$ which does not change the structure
of the surviving cells. 

Denote by $\cC_\mathcal{J}$ the $2$-full $2$-subcategory of $\cC$ formed by all $1$-morphisms in $\mathcal{J}$ together
with their respective identity $1$-morphisms. By restriction, $\mathbf{M}$ becomes a $2$-representation 
$\mathbf{M}_\mathcal{J}$ of
$\cC_\mathcal{J}$. As the additive closure of $1$-morphisms in $\mathcal{J}$ is stable with respect to left
multiplication by $1$-morphisms in $\cC$, it follows that $\mathbf{M}$ is a transitive $2$-representation
of $\cC_\mathcal{J}$. 

We claim that $\mathbf{M}_\mathcal{J}$ is simple transitive. Indeed, assume that 
$\mathbf{J}$ is an ideal of $\mathbf{M}$ stable with respect to the action of $\cC_\mathcal{J}$. Assume that it 
is nonzero and take any nonzero morphism $\alpha$ in it. As $\mathbf{M}$ is a simple transitive $2$-representation
of $\cC$, there exists a $1$-morphism $\mathrm{G}$ in $\cC$ such that $\mathrm{G}(\alpha)$ has an invertible
nonzero direct summand. Applying $1$-morphisms from $\cC_\mathcal{J}$ we, on the one hand, will map such an 
invertible direct summand to another invertible morphism (and since $\mathbf{M}$ is transitive there is a 
$1$-morphism $\mathrm{F}$ in $\cC_\mathcal{J}$ which does not annihilate this invertible direct summand). 
On the other hand,  $\mathrm{F}\circ \mathrm{G}$ is in $\mathcal{J}$ and hence application of it to
$\alpha$ cannot produce any invertible direct summands, a contradiction. Therefore $\mathbf{J}$ is zero.

By Theorem~\ref{thmmain}, Remark~\ref{rem14} and \cite[Theorem~13]{MM3}, 
$\mathbf{M}_\mathcal{J}$  is equivalent to a cell $2$-representation $\mathbf{C}^{\mathcal{J}}_{\mathcal{L}}$ 
of $\cC_\mathcal{J}$ where $\mathcal{L}$ is a left cell in $\mathcal{J}$. 
By \cite[Theorem~43]{MM1} any choice of $\mathcal{L}$ yields an equivalent $2$-representation. 
Set $\mathtt{i}=\mathtt{i}_{\mathcal{L}}$ and let $L$ be a simple
object in $\overline{\mathbf{C}}^{\mathcal{J}}_{\mathcal{L}}(\mathtt{i})$ which is not annihilated
by $1$-morphisms in $\mathcal{L}$. Then we can consider $L$ as an object in $\overline{\mathbf{M}}(\mathtt{i})$. 

Sending $P_{\mathbbm{1}_{\mathtt{i}}}$ to $L$ gives a $2$-natural transformation $\Phi$ 
from the $2$-representation $\mathbb{P}_{\mathtt{i}}$ of $\cC$ to $\overline{\mathbf{M}}$. 
In the notation of Subsection~\ref{s2.3}, the image of $\mathbf{N}(\mathtt{j})$ for $\mathtt{j}\in\cC$ under $\Phi$
is inside the category of projective objects in $\overline{\mathbf{M}}(\mathtt{j})$ and contains at least one
representative in each isomorphism class of indecomposable objects, see \cite[Subsection~4.5]{MM1}. 
We also have that $\mathbf{I}$  (see Subsection~\ref{s2.3}) annihilates $L$ by construction. It follows that the 
$2$-representation $\mathbf{K}$ of $\cC$ on projective objects in the categories $\overline{\mathbf{M}}(\mathtt{j})$ 
(for $\mathtt{j}\in\cC$) is equivalent to the cell $2$-representation $\mathbf{C}_{\mathcal{L}}$ of $\cC$. 
As $\mathbf{K}$ is equivalent to $\mathbf{M}$ by \cite[Theorem~11]{MM2}, we deduce that
$\mathbf{M}$ is equivalent to $\mathbf{C}_{\mathcal{L}}$. This completes the proof.
\end{proof}

\section{Examples}\label{s6}

\subsection{A non weakly fiat $2$-category $\cC_A$}\label{s6.1}

In this subsection we give an example of a non weakly fiat $2$-category $\cC_A$ for which Theorem~\ref{thmmain}
generalizes to the class of exact simple transitive $2$-representations. Taking into account the example
considered in Subsection~\ref{s2.4}, the present example is somewhat surprising.

For $A=\Bbbk[x,y]/(x^2,y^2,xy)$ consider the $2$-category $\cC_A$ as defined in Subsection~\ref{s4.1}.
Note that $A$ is local but not self-injective which implies that $\cC_A$ is not weakly fiat. Let $\mathrm{F}$ be an 
indecomposable $1$-morphism in $\cC_A$ which is not isomorphic to the identity $1$-morphism.
The defining $2$-representation of $\cC_A$ is easily seen to be equivalent to the cell
$2$-representation $\mathbf{C}_{\mathcal{L}}$ for $\mathcal{L}=\{\mathrm{F}\}$.

\begin{proposition}\label{prop18}
For  $A=\Bbbk[x,y]/(x^2,y^2,xy)$, any exact simple transitive $2$-representation of $\cC_A$ is equivalent
to a cell $2$-representation.
\end{proposition}

\begin{proof}
Let $\mathbf{M}$ be an exact simple transitive $2$-representation of $\cC_A$.  
Without loss of generality we may assume $\mathbf{M}(\mathrm{F})\neq 0$.
Then  $\mathrm{F}\circ \mathrm{F}\cong \mathrm{F}^{\oplus 3}$ and hence 
$\llbracket\mathrm{F}\rrbracket^2=3\llbracket\mathrm{F}\rrbracket$ by exactness of $\mathbf{M}(\mathrm{F})$.
Using Theorem~\ref{thm1} it is easy to check that $\llbracket\mathrm{F}\rrbracket$ is equal to
one of the following matrices:
\begin{displaymath}
M_1:=\left(\begin{array}{c}3\end{array}\right),\quad 
M_2:=\left(\begin{array}{cc}1&1\\2&2\end{array}\right),\quad 
M_3:=\left(\begin{array}{cc}1&2\\1&2\end{array}\right),\quad 
M_4:=\left(\begin{array}{ccc}1&1&1\\1&1&1\\1&1&1\end{array}\right). 
\end{displaymath}
Let $B$ and $B'$ be as in the proof of Theorem~\ref{thmmain}. 
Note that both  Lemma~\ref{lem7} and Lemma~\ref{lem9} are still 
applicable in our situation. Despite of the fact that $\cC_A$ is not weakly fiat, it is still $\mathcal{J}$-simple, 
where $\mathcal{J}=\{\mathrm{F}\}$.

If $\llbracket\mathrm{F}\rrbracket=M_4$, then $B\cong \Bbbk^{\oplus 3}$ and
$\mathbf{M}(\mathrm{F})$ is the direct sum of nine copies of the identity functors  (between the three different copies
of $\Bbbk\text{-}\mathrm{mod}$). The endomorphism algebra of $\mathbf{M}(\mathrm{F})$ has thus dimension nine and
is clearly not isomorphic to $A\otimes_{\Bbbk}A^{\mathrm{op}}$. Hence this case is not possible.

If $\llbracket\mathrm{F}\rrbracket=M_3$, then $B=B'\cong \Bbbk^{\oplus 2}$ and the algebra $A$ does not inject
into $B'$. This contradicts Lemma~\ref{lem9} and hence this case is not possible either.

If $\llbracket\mathrm{F}\rrbracket=M_2$, then either $B=B'$ is a $3$-dimensional algebra which is
not local or $B\cong \Bbbk^{\oplus 2}$ and $B'\cong \Bbbk\oplus\mathrm{Mat}_{2\times 2}(\Bbbk)$. 
In the first case we again get a contradiction to Lemma~\ref{lem9}.
In the second case the endomorphism algebra of $\mathbf{M}(\mathrm{F})$ has dimension ten and two
direct summands isomorphic to $\Bbbk$, say this endomorphism algebra is $Q\oplus \Bbbk\oplus \Bbbk$. 
If the local algebra $A\otimes_{\Bbbk}A^{\mathrm{op}}$ were to inject into the endomorphism algebra 
of $\mathbf{M}(\mathrm{F})$, the algebra  $A\otimes_{\Bbbk}A^{\mathrm{op}}$ would also inject into 
$Q$ which has strictly smaller dimension, a contradiction. Hence this case is not possible.

If $\llbracket\mathrm{F}\rrbracket=M_1$, then either $B\cong \Bbbk$ and $B'=\mathrm{Mat}_{3\times 3}(\Bbbk)$ 
or $B=B'$ has dimension $3$. In the former case the endomorphism algebra of $\mathbf{M}(\mathrm{F})$ has dimension 
nine and is not local, implying a contradiction similarly to the case $\llbracket\mathrm{F}\rrbracket=M_4$. 
In the latter case we again use Lemma~\ref{lem9} to get $B=B'\cong A$ and then
we readily deduce that $\mathbf{M}$ is equivalent to the cell $2$-representation.
\end{proof}

\subsection{Categorification of finite dimensional $2$-Lie algebras}\label{s6.2}

Let $\mathfrak{g}$ denote a simple finite dimensional complex Lie algebra. We fix a triangular decomposition
$\mathfrak{n}_-\oplus\mathfrak{h}\oplus\mathfrak{n}_+$ of $\mathfrak{g}$. For any $\mathfrak{h}$-weight $\lambda$ 
denote by $L(\lambda)$ the corresponding simple highest weight module with highest weight $\lambda$.
Let $\leq$ denote the natural partial order on $\mathfrak{h}$-weights. 

Let $\cU$ be the $2$-category categorifying the idempotent version $\dot{U}$ of the universal enveloping 
algebra of $\mathfrak{g}$ as defined in \cite[Definition~2.4]{We} (the origins of this $2$-category are in
\cite{CL}, see also \cite{KL,Ro} for other variations). The categorification statement is justified 
by \cite[Theorem~B.2]{We}. For each dominant integral $\mathfrak{h}$-weight $\lambda$ there is a  
$2$-representation of $\cU$ given by a functorial action on the direct sum (over $n$) 
of categories of projective modules over the cyclotomic quiver Hecke algebras (KLR algebras) $R_n^{\lambda}$ 
associated with $\mathfrak{g}$ (see \cite[Theorem~3.17]{We} for $\cU$ and also \cite{KK,Ka} for a similar statement
related to Rouquier's $2$-Kac-Moody algebras). This $2$-representation categorifies $L(\lambda)$.
We note the following properties of this $2$-representation:
\begin{itemize}
\item As $L(\lambda)$ is finite dimensional, only finitely many of the algebras $R_n^{\lambda}$ are non-zero.
\item As $L(\lambda)$ is finite dimensional, sufficiently high powers of the generators annihilate our 
$2$-representation.  Hence, the commutation relations in $\mathfrak{g}$ imply that 
only finitely many indecomposable $1$-morphisms from $\cU$ act as non-zero functors in this $2$-representation.
\item Each $R_n^{\lambda}$ is finite dimensional and all involved functors are exact.
\item Each $1$-morphism in $\cU$ acts as an exact functor and hence can be realized as tensoring
with a finite-dimensional bimodule. This implies that the spaces of two morphisms in this $2$-representation are finite dimensional.
\item Each $1$-morphism in $\cU$ has a biadjoint which is again a functor representing the action of
some $1$-morphism in $\cU$.
\item The endomorphism algebra of each identity $1$-morphism in $\cU$ is positively graded 
by the non-degeneracy part of \cite[Theorem~B.2]{We} and isomorphic to a polynomial ring (\cite[Proposition 3.31]{We}). In particular, each finite dimensional graded quotient of this algebra is local.
\end{itemize}
Let $\mathcal{I}_{\lambda}$ be the kernel of this $2$-representation and set 
$\cU_{\lambda}:=\cU/\mathcal{I}_{\lambda}$. Then the above implies that $\cU_{\lambda}$ is a fiat $2$-category.
Note that $\mathcal{I}_{\lambda}$ is, in general, not generated by $2$-morphisms of the form
$\mathrm{id}_{\mathrm{F}}$, where $\mathrm{F}$ is some $1$-morphism, but it additionally
contains some of the $2$-morphisms between $1$-morphisms which are not in $\mathcal{I}_{\lambda}$, see 
\cite[Remark~31]{MM2}.

Consider a finite set $\boldsymbol{\lambda}:=\{\lambda_1,\lambda_2,\dots,\lambda_k\}$ of 
dominant integral $\mathfrak{h}$-weights such that $\lambda_i\not\leq \lambda_j$ for all $i\neq j$
and denote by $\overline{\boldsymbol{\lambda}}$ the set of all dominant integral weights
$\mu$ such that $\mu\leq \lambda_i$ for some $i$. Note that $\overline{\boldsymbol{\lambda}}$ is a finite set. 
Define 
\begin{displaymath}
\cU_{\boldsymbol{\lambda}}:=\cU/(\mathcal{I}_{\lambda_1}\cap \mathcal{I}_{\lambda_2}\cap\dots\cap\mathcal{I}_{\lambda_k}),
\end{displaymath}
which is again a fiat $2$-category. 

\begin{remark}\label{remnew12}
{\rm
Let $\mathcal{L}$ be the left cell in $\cU_{\boldsymbol{\lambda}}$ containing the indecomposable 
$1$-morphism $\mathbbm{1}_{\lambda_l}$ for $l\in\{1,2,\dots,k\}$. As $\mathbbm{1}_{\lambda_l}$ is a genuine idempotent
and is, obviously, the unique element in the intersection of its left and right cells, the 
radical of its endomorphism ring is contained in the ideal $\mathbf{I}$ from Subsection~\ref{s2.3}
used to define the corresponding cell $2$-representation $\mathbf{C}_{\mathcal{L}}$.
Consequently, the image of $\mathbbm{1}_{\lambda_l}$ in the abelianized cell $2$-representation
is both simple and projective (this corresponds to a projective module over $R_0^{\lambda}\cong\mathbb{C}$). 
Moreover, the functor $\mathbf{C}_{\mathcal{L}}(\mathbbm{1}_{\lambda_l})$ 
is just the identity functor on the category of complex vector
spaces, in particular, its endomorphism ring consists only of scalars. 
Note that our construction of $\mathbf{C}_{\mathcal{L}}$ differs, in particular, from the construction 
of the universal categorification of $L(\lambda)$ in \cite[Subsection~5.1.2]{Ro}. In the latter case 
the endomorphism of $\mathbbm{1}_{\lambda_l}$ is much bigger in general.
}
\end{remark}

\begin{theorem}\label{thm31}
For any $\boldsymbol{\lambda}$ as above every two-sided cell in the $2$-category $\cU_{\boldsymbol{\lambda}}$ 
is strongly regular and satisfies the numerical condition.
\end{theorem}

\begin{proof}
For $l\in\{1,2,\dots,k\}$ consider the two-sided cell $\mathcal{J}$ of $\cU_{\boldsymbol{\lambda}}$ containing 
$\mathbbm{1}_{\lambda_l}$. Then, factoring out the maximal $2$-ideal 
in $\cU_{\boldsymbol{\lambda}}$ which contains $\mathrm{id}_{\mathbbm{1}_{\lambda_l}}$ and does not contain 
the identity $2$-morphism for any $1$-morphism outside $\mathcal{J}$ (note that such an ideal does not have to
be generated by $2$-morphisms of the form $\mathrm{id}_{\mathrm{F}}$, where $\mathrm{F}$ is some $1$-morphism), 
we obtain the $2$-category 
$\cU_{\boldsymbol{\mu}}$ where $\boldsymbol{\mu}$ is uniquely defined via 
$\overline{\boldsymbol{\mu}}:=\overline{\boldsymbol{\lambda}}\setminus\{\lambda_l\}$,
cf. \cite[Section~9]{DG}. Therefore
it is enough to prove that $\mathcal{J}$ is strongly regular and satisfies the numerical condition.

Let $\mathcal{L}$ denote the left cell of $\mathbbm{1}_{\lambda_l}$. 
Let further $L$ be an indecomposable object in $R_0^{\lambda}\text{-}\mathrm{proj}$.
Note that $R_0^{\lambda}\cong\mathbb{C}$.
As $L$ corresponds to the highest weight vector in $L(\lambda)$, all $1$-morphisms which do not annihilate
$L$ must correspond to the $U(\mathfrak{n}_-)$ part of $\dot{U}$. This means that $\mathcal{L}$ consists of
direct summands of powers of the negative generators of $\cU$. 
Then, from \cite[Theorem~3.17]{We} in combination with \cite[Theorem~5.7]{Ro} and 
\cite[Theorem~4.4]{VV}, it follows 
that mapping an indecomposable $1$-morphism $\mathrm{F}\in\mathcal{L}$ to 
$\mathrm{F}\, L$ induces a bijection between $\mathcal{L}$
and the set of isomorphism classes of indecomposable objects in 
\begin{displaymath}
\bigoplus_{n\geq 0}R_n^{\lambda}\text{-}\mathrm{proj}. 
\end{displaymath}
Set
\begin{displaymath}
A:=\bigoplus_{n\geq 0}R_n^{\lambda}\quad\text{ and }\quad B:=\bigoplus_{n\geq 1}R_n^{\lambda}. 
\end{displaymath}
For any $M\in B\text{-}\mathrm{proj}$ we have $\mathbbm{1}_{\lambda_l}\, M=0$ and therefore
$\mathrm{F}\, M=0$ for any $\mathrm{F}\in\mathcal{L}$. Consider 
the abelian $2$-representation $\overline{\mathbf{C}}_{\lambda_l}$.

Since $\cU_{\boldsymbol{\lambda}}$ is fiat, Lemma~\ref{lem8} implies  that 
$\overline{\mathbf{C}}_{\lambda_l}(\mathrm{F})$ is an indecomposable projective functor from
$\mathbb{C}\text{-}\mathrm{mod}$ to $A\text{-}\mathrm{mod}$. Consequently, for any 
$\mathrm{G}\in\mathcal{L}$ the functor $\overline{\mathbf{C}}_{\lambda_l}(\mathrm{F}\circ \mathrm{G}^*)$
is indecomposable. We claim that this implies that $\mathrm{F}\circ \mathrm{G}^*$ is indecomposable.
Indeed, if  $\mathrm{F}\circ \mathrm{G}^*\cong\mathrm{X}\oplus\mathrm{Y}$,  then without loss of generality
we may assume $\overline{\mathbf{C}}_{\lambda_l}(Y)=0$. Since $\mathcal{J}$ is a maximal two-sided cell,
we have $Y\in \mathcal{J}$ and hence $\overline{\mathbf{C}}_{\lambda_l}(Y)\neq 0$, a contradiction.

The previous paragraph shows that the set $\{\mathrm{F}\circ \mathrm{G}^*\}$, where 
$\mathrm{F},\mathrm{G}\in\mathcal{L}$, consists of indecomposable $1$-morphisms and hence coincides with $\mathcal{J}$.
In particular $|\mathcal{J}|=|\mathcal{L}|^2$.  It is now obvious that the left cells in $\mathcal{J}$ are 
obtained fixing $\mathrm{G}$ and the right cells in $\mathcal{J}$ are  obtained fixing $\mathrm{F}$. 
Therefore $\mathcal{J}$ is strongly regular. To check the numerical condition we note that 
$\overline{\mathbf{C}}_{\lambda_l}$ realizes elements of $\mathcal{J}$ as tensoring with indecomposable projective
$A\text{-}A$-bimodules, so the numerical condition follows from \cite[Subsection~7.3]{MM1}.
\end{proof}

\subsection{Soergel bimodules in type $B_2$}\label{s6.3}

Consider the $2$-category $\cS$ of Soergel bimodules for a Lie algebra of type $B_2$, see \cite[Section~7.1]{MM1} 
and \cite[Example~20]{MM2}. We denote by $\clubsuit$ the (unique) object in $\cS$. 
The Weyl group in this case is given by
\begin{displaymath}
W=\{e,s,t,st,ts,sts,tst,stst=tsts\},
\end{displaymath}
where $s^2=t^2=e$, and is isomorphic to the dihedral group $D_4$. The group $D_4$ has five simple
modules over $\mathbb{C}$: the one-dimensional simple modules $V_{\varepsilon,\delta}$, for
$\varepsilon,\delta\in\{\pm 1\}$, where $s$ acts via $\varepsilon$ and $t$ acts via $\delta$;
and the $2$-dimensional simple module $V_2$ (the defining geometric representation). 
For an additive category $\mathcal{A}$ we denote by $K_0(\mathcal{A})$ the split Grothendieck group of $\mathcal{A}$.
Our aim in this section is to apply previous results in order to prove the following statement 
which describes simple $W$-modules admitting a finitary categorification.

\begin{proposition}\label{propb2}
Let $\mathbf{M}$ be a finitary $2$-representation of $\cS$. Assume that the induced action of
the algebra $\mathbb{C}\otimes_{\mathbb{Z}}K_0(\cS(\clubsuit,\clubsuit))$ on the vector space
$\mathbb{C}\otimes_{\mathbb{Z}}K_0(\mathbf{M}(\clubsuit))$ gives a simple $W$-module $V$. Then
$V\cong V_{1,1}$ or $V\cong V_{-1,-1}$.
\end{proposition}

\begin{proof}
We have three two-sided cells
\begin{displaymath}
\mathcal{J}_1=\mathcal{L}_1=\{e\},\quad \mathcal{J}_2=\{s,t,st,ts,sts,tst\}, \quad \mathcal{J}_3=
\mathcal{L}_3=\{stst\}
\end{displaymath}
and $J_2$ splits into two left cells
\begin{displaymath}
\mathcal{L}_2^{(1)}=\{s,st,sts\}\quad\text{ and }\quad  \mathcal{L}_2^{(2)}=\{t,ts,tst\}.
\end{displaymath}
Right cells are obtained using the map $w\mapsto w^{-1}$. 

It is easy to check that the cell $2$-representations $\mathbf{C}_{\mathcal{L}_1}$ and
$\mathbf{C}_{\mathcal{L}_3}$ categorify $V_{1,1}$ and $V_{-1,-1}$, respectively.

We identify indecomposable Soergel bimodules $\theta_w$
for $w\in W$ with the corresponding elements
\begin{gather*}
\theta_e=e,\quad \theta_s=e+s,\quad \theta_t=e+t,\quad \theta_{st}=e+t+s+st,\quad \theta_{ts}=e+t+s+ts,\\
\quad \theta_{sts}=e+t+s+ts+st+sts,\quad \theta_{tst}=e+t+s+ts+st+tst,\\
\quad \theta_{stst}=e+t+s+ts+st+tst+sts+stst
\end{gather*}
in the Kazhdan-Lusztig basis for $\mathbb{Z}[W]$. 

Note that the element $\theta_s$ annihilates $V_{-1,1}$ while $\theta_t$ does not annihilate $V_{-1,1}$.
If we had a $2$-representation $\mathbf{M}$ decategorifying to $V_{-1,1}$, then 
$\mathbf{M}(\theta_s)=0$ while $\mathbf{M}(\theta_t)\neq 0$ which is impossible as $\theta_s$ and $\theta_t$
belong to the same two-sided cell. Therefore $V\not\cong V_{-1,1}$ and, by symmetry,
$V\not\cong V_{1,-1}$. (This argument came up in discussion with Catharina Stroppel.)

It is left to show that $V\not\cong V_2$. Note that $\theta_{stst}$ annihilates $V_2$.
Assume that $\mathbf{M}$ is a $2$-representation of $\cS$ decategorifying to $V_2$ and consider
$\overline{\mathbf{M}}$.
Set $\Theta:=\sum_{w\in J_2}\theta_w$. Direct computation shows that
\begin{displaymath}
(\theta_{st}+\theta_{ts})^2=2\Theta\mod J_3,\qquad \Theta^2=10\Theta+4(\theta_{st}+\theta_{ts})\mod J_3.
\end{displaymath}
This implies that the matrix $X:=\llbracket \theta_{st}+\theta_{ts}\rrbracket$ satisfies the polynomial
equation $X^4-20X^2-16X=0$. Consequently, $X$ is diagonalizable with eigenvalues in 
$\{0,-4,2(1\pm\sqrt{2})\}$. Clearly, $X$ is not the zero matrix. As all entries of $X$ are non-negative, 
the trace of $X$ is non-negative which implies that the eigenvalues of $X$ are
$2(1\pm\sqrt{2})$, each with multiplicity one. Thus the trace of $X$ is $4$ and the determinant is $-4$,
leaving
\begin{gather*}
\left(\begin{array}{cc}4&4\\1&0\end{array}\right),\quad
\left(\begin{array}{cc}4&2\\2&0\end{array}\right),\quad
\left(\begin{array}{cc}4&1\\4&0\end{array}\right),\quad
\left(\begin{array}{cc}3&7\\1&1\end{array}\right),\quad
\left(\begin{array}{cc}3&1\\7&1\end{array}\right),\\
\left(\begin{array}{cc}2&8\\1&2\end{array}\right),\quad
\left(\begin{array}{cc}2&4\\2&2\end{array}\right),\quad
\left(\begin{array}{cc}2&2\\4&2\end{array}\right),\quad
\left(\begin{array}{cc}2&1\\8&2\end{array}\right)
\end{gather*}
as possibilities (up to reordering of the basis).

We have $\theta_s^2\cong 2\theta_s$ and $\theta_t^2\cong 2\theta_t$, which implies that both 
$\llbracket \theta_{s}\rrbracket$ and $\llbracket \theta_{t}\rrbracket$ satisfy the polynomial
equation $x^2-2x=0$. Similarly to the above, this leads to the list of candidates for
$\llbracket \theta_{s}\rrbracket$ and $\llbracket \theta_{t}\rrbracket$ being given by
\begin{displaymath}
\left(\begin{array}{cc}1&1\\1&1\end{array}\right),\quad
\left(\begin{array}{cc}2&a\\0&0\end{array}\right),\quad
\left(\begin{array}{cc}2&0\\a&0\end{array}\right),\quad
\left(\begin{array}{cc}2&0\\0&2\end{array}\right)
\end{displaymath}
where $a\in\{0,1,2,\dots\}$. Note that $\theta_{st}=\theta_{s}\theta_{t}$ and $\theta_{ts}=\theta_{t}\theta_{s}$.
Hence, the equation 
\begin{displaymath}
\llbracket \theta_{st}+\theta_{ts}\rrbracket = \llbracket \theta_{s}\rrbracket\llbracket \theta_{t}\rrbracket +\llbracket \theta_{t}\rrbracket\llbracket \theta_{s}\rrbracket 
\end{displaymath}
reduces the choice to 
\begin{equation}\label{eqbb2}
\llbracket \theta_{s}\rrbracket=
\left(\begin{array}{cc}1&1\\1&1\end{array}\right),\quad
\llbracket \theta_{t}\rrbracket=
\left(\begin{array}{cc}2&0\\0&0\end{array}\right)
\end{equation}
or
\begin{equation}\label{eqbb2-1}
\llbracket \theta_{s}\rrbracket=
\left(\begin{array}{cc}2&0\\1&0\end{array}\right),\quad
\llbracket \theta_{t}\rrbracket=
\left(\begin{array}{cc}0&2\\0&2\end{array}\right)
\end{equation} 
or vice versa.

In case of \eqref{eqbb2}, we may restrict $\mathbf{M}$ to the $2$-subcategory ${\cT}$ of $\cS$ generated by 
$\theta_e$ and $\theta_s$
and adjunction morphisms between them. This $2$-category clearly satisfies all hypotheses of Theorem~\ref{thm15}.
Note that $\theta_s$ is self-adjoint, hence Lemma~\ref{lem6} implies 
that this restricted $2$-representation is transitive. Let $\mathbf{N}$ be its simple transitive quotient. 
Then $\mathbf{N}$ gives rise to a simple transitive $2$-representation of ${\cT}$ in which $\theta_s$ has the 
matrix described by \eqref{eqbb2}. This, however, contradicts Theorem~\ref{thm15}. 

In case of \eqref{eqbb2-1}, consider $\overline{\mathbf{M}}$
and let $L_1$ and $L_2$ denote the simple objects
in $\overline{\mathbf{M}}(\clubsuit)$.
Note that, by adjunction, a simple object $L$ can
appear in the top or in the socle of some $\theta_x N$,
for $x\in\{s,t\}$, only if $\theta_x\, L\neq 0$.
Therefore $\theta_s L_1$ cannot have $L_2$ in top or
socle and hence is uniserial of Loewy length three with
simple top and simple socle isomorphic to $L_1$. Similarly,
the module $\theta_t L_2$ has simple top and simple socle
isomorphic to $L_2$. Let $M$ denote the homology in the
middle term of
\begin{displaymath}
0\to L_2\to  \theta_t L_2\to L_2\to 0.
\end{displaymath}
Then  $M$ has length two with both simple subquotients
isomorphic to $L_1$.

Assume $M\cong L_1\oplus L_1$. Let $N$ be a non-split
extension of length two with top $L_1$ and socle $L_2$.
Then, since $\theta_s L_2=0$ and $\theta_s$ is exact,
by adjunction we have
\begin{displaymath}
\dim \mathrm{Hom}(\theta_s L_1, N)=
\dim \mathrm{Hom}(L_1, \theta_s L_1)=1
\end{displaymath}
and thus $N$ is a quotient of $\theta_s L_1$.
This implies  $\dim\mathrm{Ext}^1(L_1,L_2)=1$.
At the same time, consider $\mathrm{Rad}(\theta_t L_2)$.
By the above, this has simple socle $L_2$, the quotient
over which is $M$. Hence $\dim\mathrm{Ext}^1(L_1,L_2)\geq 2$,
a contradiction.

Assume that $M$ is indecomposable. Then, by adjunction,
\begin{displaymath}
1=\dim \mathrm{Hom}(M, \theta_s L_1)=
\dim \mathrm{Hom}(\theta_s M, L_1).
\end{displaymath}
As only $L_1$ can be in the top of $M$, the module
$\theta_s M$ has simple top and hence
is indecomposable. Consequently,
$\theta_s \theta_t \theta_s L_1\cong \theta_s M$ is
indecomposable. However,
$\theta_s \theta_t \theta_s =\theta_{sts}\oplus \theta_s$
and thus $\theta_s M$ has $\theta_s L_1$
as a direct summand. As dimensions of $\theta_s M$ and
$\theta_s L_1$ are different, this is a contradiction.
The proof is complete. 
\end{proof}


\noindent
Volodymyr Mazorchuk, Department of Mathematics, Uppsala University,
Box 480, 751 06, Uppsala, SWEDEN,\\ {\tt mazor\symbol{64}math.uu.se};
http://www.math.uu.se/$\sim$mazor/.

\noindent
Vanessa Miemietz, School of Mathematics, University of East Anglia,\\
Norwich NR4 7TJ, UK, \\ {\tt v.miemietz\symbol{64}uea.ac.uk};
http://www.uea.ac.uk/$\sim$byr09xgu/.

\begin{thebibliography}{99999}
\bibitem[AM]{AM} T.~Agerholm, V.~Mazorchuk. On selfadjoint functors 
satisfying polynomial relations. J. Algebra. {\bf 330} (2011), 448--467.
\bibitem[BFK]{BFK} J.~Bernstein, I.~Frenkel, M.~Khovanov. A categorification of the Temperley-Lieb algebra and Schur 
quotients of $U(\mathfrak{sl}_2)$ via projective and Zuckerman functors. Selecta Math. (N.S.) {\bf 5} (1999), no. 
2, 199--241. 
\bibitem[BG]{BG} J.~Bernstein, S.~Gelfand. Tensor products of 
finite- and infinite-dimensional representations of semisimple Lie 
algebras.  Compositio Math.  {\bf 41}  (1980), no. 2, 245--285. 
\bibitem[CL]{CL} S.~Cautis, A.~Lauda. Implicit structure in $2$-representations of quantum groups.
Preprint arXiv:1111.1431 
\bibitem[CR]{CR} J.~Chuang, R.~Rouquier. Derived equivalences for 
symmetric groups and $\mathfrak{sl}_2$-ca\-te\-go\-ri\-fi\-ca\-ti\-on. 
Ann. of Math.  (2) {\bf 167} (2008), no. 1, 245--298. 
\bibitem[DG]{DG} S.~Doty, A.~Giaquinto. Cellular bases of generalized $q$-Schur algebras. 
Preprint arXiv:1012.5983v3.
\bibitem[EW]{EW} B.~Elias, G.~Williamson. The Hodge theory of Soergel bimodules.
Preprint arXiv:1212.0791. To appear in Ann. of Math.
\bibitem[Fr]{Fr} P.~Freyd. Representations in abelian categories. 
Proc. Conf. Categorical Algebra (1966), 95--120.
\bibitem[Fro1]{Fro1} 
G.~Frobenius. {\"U}ber Matrizen aus positiven Elementen, 1. Sitzungsber. K{\"o}nigl. Preuss. Akad. Wiss.
(1908), 471--476.
\bibitem[Fro2]{Fro2} 
G.~Frobenius. {\"U}ber Matrizen aus positiven Elementen, 2. Sitzungsber. K{\"o}nigl. Preuss. Akad. Wiss.
(1909), 514--518.
\bibitem[GM]{GM} 
O.~Ganyushkin, V.~Mazorchuk. Classical finite transformation semigroups. An introduction. Algebra and 
Applications, {\bf 9}. Springer-Verlag London, Ltd., London, 2009.
\bibitem[KK]{KK} S.-J.~Kang, M.~Kashiwara. Categorification of highest weight modules 
via Khovanov-Lauda-Rouquier algebras. Invent. Math. {\bf 190} (2012), no. 3, 699--742.
\bibitem[Ka]{Ka} M.~Kashiwara. Biadjointness in cyclotomic Khovanov-Lauda-Rouquier algebras. 
Publ. Res. Inst. Math. Sci. {\bf 48} (2012), no. 3, 501--524.
\bibitem[KL]{KL} M.~Khovanov, A.~Lauda. A categorification of a
quantum $\mathfrak{sl}_n$. Quantum Topol. {\bf 1} (2010), 1--92.
\bibitem[Le]{Le}  T.~Leinster. Basic bicategories. Preprint arXiv:math/9810017.
\bibitem[McL]{McL} S.~Mac Lane. Categories for the working mathematician. 
Second edition. Graduate Texts in Mathematics, {\bf 5}. Springer-Verlag, New York, 1998.
\bibitem[MM1]{MM1} V.~Mazorchuk, V.~Miemietz. Cell $2$-representations of finitary
$2$-categories; Compositio Math. {\bf 147} (2011), 1519--1545.
\bibitem[MM2]{MM2} V.~Mazorchuk, V.~Miemietz. Additive versus abelian $2$-representations of 
fiat $2$-ca\-te\-go\-ri\-es. Moscow Math. J. {\bf 14} (2014), no. 3, 595--615.
\bibitem[MM3]{MM3} V.~Mazorchuk, V.~Miemietz. Endomorphisms of cell $2$-representations. 
Preprint arXiv:1207.6236.
\bibitem[MM4]{MM4} V.~Mazorchuk, V.~Miemietz. Morita theory for finitary $2$-categories. 
Preprint arXiv:1304.4698. To appear in Quantum Topol.
\bibitem[Me]{Me}
C.~Meyer. Matrix analysis and applied linear algebra. Society for Industrial and Applied Mathematics (SIAM), Philadelphia, PA, 2000.
\bibitem[Pe]{Pe} 
O.~Perron. Zur Theorie der Matrices. Mathematische Annalen {\bf 64} (2) (1907), 248--263.
\bibitem[Ro1]{Ro} R.~Rouquier. $2$-Kac-Moody algebras. Preprint arXiv:0812.5023. 
\bibitem[Ro2]{Ro2} R.~Rouquier. Quiver Hecke algebras and $2$-Lie algebras. Algebra Colloq. 
{\bf 19} (2012), no. 2, 359--410.
\bibitem[So]{So} W.~Soergel. The combinatorics of Harish-Chandra bimodules. 
J. Reine Angew. Math. {\bf 429} (1992), 49--74. 
\bibitem[VV]{VV} M.~Varagnolo, E.~Vasserot. Canonical bases and KLR-algebras. J. Reine Angew. Math. 
{\bf 659} (2011), 67--100.
\bibitem[We]{We} B.~Webster. Knot invariants and higher representation theory. Preprint arXiv:1309.3796
\bibitem[Xa]{Xa} Q.~Xantcha. Gabriel $2$-Quivers for Finitary $2$-Categories.
Preprint arXiv:1310.1586.
\end{thebibliography}
\end{document}